\documentclass[preprint,hidelinks,onefignum,onetabnum]{siamart220329}

\usepackage{lipsum}
\usepackage{amsfonts}
\usepackage{graphicx}
\usepackage{epstopdf}
\usepackage{algorithmic}
\ifpdf
  \DeclareGraphicsExtensions{.eps,.pdf,.png,.jpg}
\else
  \DeclareGraphicsExtensions{.eps}
\fi

\headers{Density compression with non-negative tensor train}{X. Tang, R. Dwaraknath, L. Ying}

\title{Variational inference and density estimation with non-negative tensor train}

\author{Xun Tang\thanks{Correspondence author, Department of Mathematics, Stanford University, Stanford, CA 94305, USA. 
  (\email{xuntang@stanford.edu}).}
  \and
  Rajat Dwaraknath\thanks{Institute for Computational and Mathematical Engineering (ICME), Stanford University, Stanford, CA 94305, USA. 
  (\email{rajatvd@stanford.edu}).}
  \and
  Lexing Ying\thanks{Department of Mathematics and Institute for Computational and Mathematical Engineering (ICME), Stanford University, Stanford, CA 94305, USA. 
  (\email{lexing@stanford.edu})}
  \funding{X.T. and L.Y. are supported by AFOSR MURI award FA9550-24-1-0254.}
  }

\usepackage{amsopn}

\makeatletter
\newcommand*{\addFileDependency}[1]{%
  \typeout{(#1)}%
  \@addtofilelist{#1}%
  \IfFileExists{#1}{}{\typeout{No file #1.}}%
}
\makeatother

\usepackage{amsmath,amssymb,a4wide}
\usepackage{latexsym}
\usepackage{indentfirst}
\usepackage{caption}
\usepackage{placeins}
\usepackage{enumerate}
\usepackage{booktabs}
\usepackage{algorithm}
\usepackage{algorithmic}
\usepackage{multirow}
\usepackage{optidef}
\usepackage{graphicx,color}
\usepackage{diagbox}
\usepackage[normalem]{ulem}
\usepackage{hyperref}
\usepackage{cleveref}
\usepackage{subcaption} %

\ifpdf
  \hypersetup{
    pdftitle={NTT-for-Density},
    pdfauthor={X. Tang, R. Dwaraknath, and L. Ying}
  }
\fi

\newcommand{\f}{\ell}
\newcommand{\G}{G}
\renewcommand{\a}{\rho}

\usepackage{amsmath,amsfonts,bm}

\def\eqref#1{equation~\ref{#1}}

\def\1{\bm{1}}

\DeclareMathAlphabet{\mathsfit}{\encodingdefault}{\sfdefault}{m}{sl}
\SetMathAlphabet{\mathsfit}{bold}{\encodingdefault}{\sfdefault}{bx}{n}

\newcommand{\R}{\mathbb{R}}
\newcommand{\C}{\mathbb{C}}

\begin{document}

\maketitle
\begin{abstract}
  This work proposes an efficient numerical approach for compressing a high-dimensional discrete distribution function into a non-negative tensor train (NTT) format.
  The two settings we consider are variational inference and density estimation, whereby one has access to either the unnormalized analytic formula of the distribution or the samples generated from the distribution.
  In particular, the compression is done through a two-stage approach.
  In the first stage, we use existing subroutines to encode the distribution function in a tensor train format.
  In the second stage, we use an NTT ansatz to fit the obtained tensor train.
  For the NTT fitting procedure, we use a log barrier term to ensure the positivity of each tensor component, and then utilize a second-order alternating minimization scheme to accelerate convergence.
  In practice, we observe that the proposed NTT fitting procedure exhibits drastically faster convergence than an alternative multiplicative update method that has been previously proposed.
  Through challenging numerical experiments, we show that our approach can accurately compress target distribution functions.
\end{abstract}



\section{Introduction}
Compressing high-dimensional distribution functions is a ubiquitous task in data science and scientific computing.
This work considers the discrete case where the distribution function is a \(d\)-dimensional tensor with non-negative entries.
Writing \([n] := \{1, \ldots, n\}\), the ground truth is a \(d\)-dimensional tensor \(P \colon [n]^{d} \to \R_{\geq 0}\).
Our goal is to represent \(P\) with a tractable parameterized family of tensors.

We consider the task under two broad settings.
The first setting is variational inference \cite{jordan1999introduction,blei2017variational}, where one has noiseless access to arbitrary entries of \(P\).
In other words, for an arbitrary multi-index \((i_1, \ldots, i_d) \in [n]^d\), the variational inference setting assumes that one can query \(P(i_1, \ldots, i_d)\) up to an unknown normalization constant.
The second setting is density estimation \cite{silverman2018density}.
In this case, we assume that one has a collection of samples \(\left(y_{1}^{(j)}, \ldots, y_{d}^{(j)}\right)_{j=1}^{N} \subset [n]^d\) which are distributed according to \(P\).
In both cases, the goal is to find the best fit to \(P\) within a given parameterized family of tensors.

In this work, we propose to use the non-negative tensor train (NTT) ansatz to compress \(P\).
The ansatz is parameterized by a collection of non-negative tensor components \(\G = (G_{k})_{k = 1}^{d}\), where \(G_{1} \in \R_{\geq 0}^{ n \times \a_{1}}\), \(G_{i} \in \R_{\geq 0}^{\a_{i-1} \times n \times \a_{i}}\) for \(i = 2,\ldots, d-1\), and \(G_{d} \in \R_{\geq 0}^{\a_{d-1} \times n }\).
The ansatz parameterized by \(\G\) takes the following form:
\begin{equation}\label{eqn: def of TT}
  \begin{aligned}
    P_{\G}(i_{1}, \ldots, i_{d}) =\sum_{\alpha_{1}, \ldots, \alpha_{d-1}}G_{1}(i_{1}, \alpha_{1})G_{2}(\alpha_{1}, i_{2}, \alpha_{2})\cdots G_{d}(\alpha_{d-1}, i_{d}).
  \end{aligned}
\end{equation}
In particular, the non-negativity of \(G_k\) ensures that \(P_{\G}\) only has non-negative entries.
Once one obtains \(P_{\G}\) as an approximation of \(P\), one can use \(P_{\G}\) to perform efficient moment estimation and efficient sampling.
One can also use \(P_{\G}\) to calculate marginal distributions and conditional distributions.

One can see that the compression of \(P\) in an NTT ansatz is a generalization of non-negative matrix factorization (NMF) \cite{lee1999learning, lee2000algorithms, wang2012nonnegative} to multi-dimensional arrays.
When \(d = 2\), the formula in \Cref{eqn: def of TT} is exactly an NMF ansatz.
For general cases, NMF is known to be NP-hard \cite{vavasis2010complexity}. Therefore, our goal for NTT compression is to obtain a good empirical algorithm rather than rigorous convergence guarantees.
For NTT factorization, the ground truth tensor \(P\) is exponential-sized, and the downstream compression task is considerably more challenging than the NMF case.
For the variational inference setting and the density estimation setting, we introduce a compression algorithm with an \(O(d)\) computational complexity. Therefore, the proposed NTT compression scheme avoids the curse of dimensionality.

\subsection{Main contribution}
We go over the main idea for the NTT compression proposal.
This work takes a two-stage approach to perform NTT compression.
In the first stage, we perform a tensor train (TT) approximation of \(P\).
The tensor train ansatz is defined in \Cref{eqn: def of TT} if one relaxes entries of \(G_k\) to be defined over \(\R\) instead of \(\R_{\geq 0}\).
In the variational inference setting, we use the TT-cross method \cite{oseledets2010tt} to obtain a TT approximation of \(P\).
In the density estimation setting, we use the TT-sketch \cite{hur2023generative} method to obtain a TT approximation of \(P\).
When \(P\) admits a TT ansatz, the output of both methods robustly converges to \(P\) \cite{qin2022error,tang2023generative}.

From either TT-cross or TT-sketch, the output of the first stage is a tensor train ansatz \(\tilde{P}\) so that \(\tilde{P} \approx P\).
In the second stage, we use an NTT ansatz \(P_{\G}\) to fit \(\tilde{P}\).
Let \(\G = (G_1, \ldots, G_d)\) be the collection of non-negative tensor components.
The NTT fitting task is defined over a minimization task where the loss function is
\begin{equation}\label{eqn: loss function ver 1}
  \begin{aligned}
    \f(\G) = & \lVert P_{\G} - \tilde{P} \rVert_{F}^2 - \mu_1 \sum_{i_{1}, \alpha_{1}} \log\left(G_1(i_1, \alpha_{1})\right)  - \sum_{k = 2}^{d-1}\mu_k \sum_{\alpha_{k-1}, i_k, \alpha_{k}} \log\left(G_k(\alpha_{k-1}, i_k, \alpha_k)\right) \\
    -        & \mu_d \sum_{\alpha_{d-1}, i_{d}} \log\left(G_d(\alpha_{d-1}, i_{d})\right),
  \end{aligned}
\end{equation}
where \(\mu_{k}\) is the regularization strength parameter on \(G_k\).
The log barrier term in \Cref{eqn: loss function ver 1} automatically ensures the positivity of all entries for each \(G_k\).

We detail how to optimize over the loss function \(\f\) in \Cref{eqn: loss function ver 1}.
As is typical in tensor network optimization, we employ an alternating minimization strategy, where we optimize one tensor component \(G_k\) at a time.
The squared Frobenius loss \(\lVert P_{\G} - \tilde{P} \rVert_{F}^2\) is quadratic as a function of \(G_k\).
Therefore, the loss function \(\f\) in \Cref{eqn: loss function ver 1} is strongly convex when viewed as a function of \(G_k\). The optimization over \(G_k\) is performed using a second-order Newton step. The \(\mu_k\) parameters are decreased, typically exponentially, so that the log barrier term in \(\f\) has a small coefficient in the later stage of the optimization.

The NTT fitting proposal under alternating minimization and log barrier shows remarkable empirical success.
Through challenging numerical instances in \Cref{sec: numerics}, we show that the NTT ansatz \(P_{\G}\) approximates \(\tilde{P}\) quite accurately.
In some cases, we observe that the squared Frobenius loss can achieve machine accuracy.
A similar two-stage approach has been proposed in \cite{shcherbakova2019nonnegative} with a first-order multiplicative update method for NTT fitting. We use the method in \cite{shcherbakova2019nonnegative} as a benchmark, and we see that it converges quite slowly when compared to our approach.
We test the first-order method and demonstrate that its accuracy is quite low, even after a substantial number of iterations.

\subsection{Related work}

\paragraph{Non-negative matrix factorization}
Non-negative matrix factorization (NMF) has been extensively explored as a subroutine in data analysis \cite{lee2000algorithms, wang2012nonnegative}.
NMF is used in finding intrinsic objects in image data and in finding semantic features of text data \cite{lee1999learning, donoho2003does}.
It is well-known that NMF is NP-hard \cite{vavasis2010complexity}.
The Lee-Seung algorithm in \cite{lee2000algorithms} suggests two methods to perform multiplicative updates on the matrix factors, and they respectively enjoy a monotonic improvement guarantee on the Frobenius norm loss and the Kullback-Leibler divergence loss.
Subsequent works have used multiplicative updates based on other divergence metrics \cite{cichocki2008non, cichocki2011generalized}.
Alternating minimization with second-order optimization has been considered in the NMF setting with quasi-Newton steps \cite{zdunek2007nonnegative} and non-negative least-squares steps \cite{kim2007fast, berry2007algorithms}.
In the NMF setting, it is well-known that multiplicative update methods are slowly convergent \cite{zdunek2007nonnegative}.
However, multiplicative update methods have considerably cheaper per-iteration complexity than the second-order methods.

\paragraph{Tensor train approximation}
Tensor train (TT) \cite{oseledets2011tensor} is a tensor network with a one-dimensional structure.
This ansatz is known in the physics community as the matrix product state under open boundary conditions \cite{white1992density}.
A high-dimensional tensor \(P\) is well-represented by a tensor train if it satisfies a series of approximate low-rank conditions for its unfolding matrices \cite{schneider2014approximation, tang2025tensor}.
When the ground truth \(P\) has a TT ansatz, running the TT-cross and TT-sketch algorithm outputs a TT ansatz \(\tilde{P}\) so that \(\tilde{P} \approx P\).
The convergence guarantee is given in the infinity norm for TT-sketch \cite{tang2023generative,tang2025initialization}, and a similar guarantee is made in the Frobenius norm for TT-cross \cite{qin2022error}.
Despite the recovery guarantee, for compressing high-dimensional distribution tensors, neither TT-cross nor TT-sketch can guarantee that the obtained ansatz has only non-negative entries.
Therefore, when non-negativity is required, it is more desirable to use a non-negative tensor ansatz to fit the obtained TT ansatz.

\subsection{Outline}
This work is organized as follows.
\Cref{sec: main alg} details the proposed NTT compression subroutine for variational inference and density estimation. \Cref{sec: self-concordance} analyzes the convergence of the Newton step through a self-concordance analysis. \Cref{sec: numerics} shows the performance of the proposed ansatz on practical variational inference and density estimation. \Cref{sec: conclusion} gives concluding remarks.

\section{Main algorithm}\label{sec: main alg}
This section goes through the non-negative tensor train compression subroutine for variational inference and density estimation.
\Cref{sec: tensor-train} gives background information on the tensor train ansatz and covers the first stage of NTT compression, where a tensor train compression \(\tilde{P}\) is given to approximate \(P\). \Cref{sec: NTT fitting} presents the second stage of NTT compression where an NTT ansatz \(P_{\G}\) is used to approximate \(\tilde{P}\). \Cref{sec: NTT fitting extra} details the numerical techniques used to accelerate the Newton step during each alternating minimization step under NTT fitting.

\paragraph{Motivating example}

One prevalent type of discrete distribution tensors comes from a grid-based discretization of unnormalized continuous Boltzmann distribution functions.
We consider an unnormalized distribution function \(p \colon [-a,a]^{d} \to \R_{\geq 0}\) and let \((x_{i})_{i = 1}^{n}\subset [-a, a]\) be a choice of grid points on \([-a, a]\).
One can construct a discrete distribution tensor \(P\) by setting
\[
  P(i_1, \ldots, i_d) = p(x_{i_1}, \ldots, x_{i_d}).
\]

One example of \(p\) is the Ginzburg-Landau model, which is used for studying the phenomenological theory of superconductivity \cite{weinan2004minimum, ginzburg2009theory}.
We take the 1D setting, and the associated unnormalized distribution function is
\[
  p(z_1, \ldots, z_d) = \exp\left(-\frac{\gamma}{2} \sum_{i = 1}^{d-1} \left(z_{i} - z_{i+1}\right)^2 - \frac{\lambda}{2} \sum_{i = 1}^{d} \left(1 - z_i^2\right)^2\right),
\]
where \(\gamma\) controls the correlation strength and \(\lambda\) controls the strength of the double-well term.
One can define \(P\) accordingly from \(p\).
In the context of variational inference, evaluating entries of \(P\) can be achieved by evaluating \(p\); thus, accessing individual entries of \(P\) is typically inexpensive.
Samples of \(P\) can be obtained by Monte-Carlo Markov chain (MCMC) simulation based on the analytic formula of \(P\) \cite{liu2001monte}.
Once the samples are obtained, they can be used for density estimation.

\subsection{Stage one: tensor train compression}\label{sec: tensor-train}
We go through the first stage where one obtains a tensor train approximation \(\tilde{P}\) from either the analytic formula of \(P\) or the samples of \(P\).
This subsection is self-contained, and further technical details can be found in \cite{oseledets2010tt,hur2023generative}.

\paragraph{Notation}
For notational compactness, we introduce several shorthand notations for simple derivations.
For an index set \(S \subset [d]\), we let \(i_{S}\) stand for the subvector with entries from the index set $S$.

\paragraph{Tensor train}
We define a tensor train (TT) ansatz \(P\colon [n]^{d} \to \mathbb{R}\) in terms of tensor components \((F_{k})_{k = 1}^{d}\), where \(F_{1} \in \R^{ n \times r_{1}}\), \(F_{i} \in \R^{r_{i-1} \times n \times r_{i}}\) for \(i = 2,\ldots, d-1\), and \(F_{d} \in \R^{r_{d-1} \times n }\).
Using the multi-index notation, the equation for \(P\) is written as follows:
\begin{equation}\label{eq: TT}
  P(i_{[d]}) =\sum_{\alpha_{[d-1]}}F_{1}(i_{1}, \alpha_{1})F_{2}(\alpha_{1}, i_{2}, \alpha_{2})\cdots F_{d}(\alpha_{d-1}, i_{d}).
\end{equation}
The TT ansatz is versatile, and this representation does not suffer from the curse of dimensionality for many important tasks.
One can see from \Cref{eq: TT} that evaluating a single entry of \(P\) has an \(O(d)\) complexity.
To calculate the normalization constant \(\sum_{i_{[d]}}P(i_{[d]})\), one computes
\begin{equation}\label{eq: TT normalization constant}
  \sum_{i_{[d]}}
  P(i_{[d]}) =\sum_{\alpha_{[d-1]}}\left(\sum_{i_{1}}F_{1}(i_{1}, \alpha_{1})\right)\left( \sum_{i_2}F_{2}(\alpha_{1}, i_{2}, \alpha_{2})\right)\cdots \left(\sum_{i_d} F_{d}(\alpha_{d-1}, i_{d})\right).
\end{equation}
One can see that the expression in \Cref{eq: TT normalization constant} can be computed by a sequence of matrix-vector multiplications.

\paragraph{Sketched-linear equation for \(F_k\)}
Both the TT-cross algorithm and the TT-sketch algorithm can be derived by performing sketching.
The derivation is done in \cite{tang2025wavelet} for any tensor network defined over general tree structures.
For the reader's convenience, we go through the derivation for the tensor train case.

We detail how to form a sketched linear equation for the tensor component \(F_{k}\).
For simplicity, we first assume that \(k \not \in \{1, d\}\) so that \(k\) is not a boundary node.
We assume that the ground truth tensor \(P\) admits a tensor train ansatz.
Then, \Cref{eq: TT} implies that there exist two tensors \(P_{>k} \colon [r_{k}] \times [n]^{d-k} \to \R, P_{<k} \colon [n]^{k-1} \times [r_{k-1}] \to \R\), such that the following equation holds:
\begin{equation}\label{eqn: low-rank}
  P(i_{1}, \ldots, i_{d}) = \sum_{\alpha_{k-1}, \alpha_{k}}P_{<k}(i_{[k-1]}, \alpha_{k-1})F_{k}(\alpha_{k-1}, i_{k}, \alpha_{k})P_{>k}(\alpha_{k}, i_{[d] - [k]}).
\end{equation}

One can see that \Cref{eqn: low-rank} is an exponentially-sized over-determined system of linear equations on \(F_{k}\).
Therefore, it is natural to perform sketching to have a tractable system of linear equations for \(F_{k}\).
We let \(\tilde{r}_{k-1}, \tilde{r}_{k}\) be two integers such that \(\tilde{r}_{k-1} \geq r_{k-1}\) and \(\tilde{r}_{k} \geq r_{k}\).
We introduce two sketch tensors:
\(S_{>k} \colon [n]^{d-k} \times [\tilde{r}_{k}] \to \R, S_{<k} \colon [\tilde{r}_{k-1}] \times [n]^{k-1} \to \R\).
By contracting \Cref{eqn: low-rank} with \(S_{>k} \) and \(S_{<k}\) respectively on the \(i_{[d] - [k]}\) and \(i_{[k-1]}\) index, one obtains
\begin{equation}\label{eqn: sketched linear system for F}
  \sum_{\alpha_{k-1}, \alpha_{k}}A_{<k}(\zeta_{k-1}, \alpha_{k-1})F_{k}(\alpha_{k-1}, i_{k}, \alpha_{k})A_{>k}(\alpha_{k}, \omega_{k}) = B_{k}(\zeta_{k-1}, i_k, \omega_{k}),
\end{equation}
where \(A_{>k}, A_{<k}\) are respectively the contraction of \(P_{>k}, P_{<k}\) by \(S_{>k}, S_{<k}\) and \(B_{k}\) is the contraction of \(P\) by \(S_{>k} \otimes S_{<k} \).
By writing \(A_{k} = A_{>k} \otimes A_{<k}\) and by appropriate permutation and grouping of indices, one can see that \Cref{eqn: sketched linear system for F} can be viewed as a linear system \(A_{k}F_{k} = B_{k}\).
Thus, after one obtains \(A_{>k}, A_{<k}\) and \(B_{k}\), one can solve for \(F_{k}\).

For the boundary case of \(k = 1\) and \(k = d\), there exist \(P_{>1}\) and \(P_{<d}\) such that the following equation holds:
\begin{equation}
  P(i_{1}, \ldots, i_{d}) = \sum_{\alpha_{1}}F_{1}(i_{1}, \alpha_{1})P_{>1}(\alpha_{1}, i_{[d] - \{1\}}), \,\,\,
  P(i_{1}, \ldots, i_{d}) = \sum_{\alpha_{d-1}}P_{<d}(i_{[d-1]}, \alpha_{d-1})F_{d}(\alpha_{d-1}, i_{d}).
\end{equation}
Likewise, one can define sketch tensors \(S_{>1}\) and \(S_{<d}\).
Contraction of with \(S_{>1}\) and \(S_{<d}\) leads to the sketched linear equation for \(F_{1}\) and \(F_{d}\) as follows:
\begin{equation}\label{eqn: low-rank boundary case}
  B_{1}(i_1, \omega_1) = \sum_{\alpha_{1}}F_{1}(i_{1}, \alpha_{1})A_{>1}(\alpha_{1}, \omega_{1}),\quad
  B_{d}(\zeta_{d-1}, i_d) = \sum_{\alpha_{d-1}}A_{<d}(\zeta_{d-1}, \alpha_{d-1})F_{d}(\alpha_{d-1}, i_{d}).
\end{equation}

\subsubsection{TT-cross algorithm for variational inference}
We cover the procedure of TT-cross for the setting of variational inference.
The input is a function handle in which one can query an arbitrarily chosen entry of \(P\).
The output is a tensor train approximation for \(P\).

\paragraph{Obtain \(B_{k}\) by \(P\) queries}
The TT-cross algorithm obtains \(B_k\) by setting \(S_{>k}\) and \(S_{<k}\) to be sparse tensors.
The sparsity of the sketch tensors reduces obtaining \(B_k\) to querying entries of \(P\).
For \(S_{>k}\) we select \(\tilde{r}_{k}\) indices \(I^{>k} =\{w^{\omega_{k}} = (w_{j}^{\omega_{k}})_{j \in [d] - [k]}\}_{\omega_{k} \in [\tilde{r}_{k}]} \subset [n]^{d-k}\).
For \(S_{<k}\) we select \(\tilde{r}_{k-1}\) indices \(I^{<k} =\{v^{\zeta_{k-1}} = (v_{j}^{\zeta_{k-1}})_{j \in [k-1]}\}_{\zeta_{k-1} \in [\tilde{r}_{k-1}]} \subset [n]^{k-1}\).
Let \(e_{l}\) be the \(l\)-th standard basis in \(\R^{n}\).
Corresponding to the selected indices, we construct the sketch tensor \(S_{>k}, S_{<k}\) to be defined by
\begin{equation}\label{eqn: def of S_k in VI}
  S_{>k}(:, \omega_{k}) = \bigotimes_{j=k+1}^{d}e_{w^{\omega_{k}}_{j}}, \quad S_{<k}(\zeta_{k-1},:) = \bigotimes_{j=1}^{k-1}e_{v^{\zeta_{k-1}}_{j}}.
\end{equation}

The construction of \(S_{>k}, S_{<k}\) in \Cref{eqn: def of S_k in VI} immediately implies the following equation for \(B_k\):
\begin{equation}\label{eqn: def of B_k in VI}
  B_{k}(\zeta_{k-1}, i_{k}, \omega_{k}) = P(v_{1}^{\zeta_{k-1}}, \ldots, v_{k-1}^{\zeta_{k-1}}, i_k, w_{k+1}^{\omega_{k}}, \ldots, w_{d}^{\omega_{k}}).
\end{equation}

When \(k = 1\) and \(k = d\), one can likewise define \(S_{>1}\) and \(S_{<d}\).
The construction of \(B_{1}\) and \(B_{d}\) are given by
\begin{equation}\label{eqn: def of B_k in VI boundary}
  B_{1}(i_{1}, \omega_{1}) = P(i_1, w_{2}^{\omega_{1}}, \ldots, w_{d}^{\omega_{1}}), \quad B_{d}(\zeta_{d-1}, i_{d}) = P(v_{1}^{\zeta_{d-1}}, \ldots, v_{d-1}^{\zeta_{d-1}}, i_d).
\end{equation}

\paragraph{Obtain \(A_k\) by \(P\) queries + SVD}
To obtain \(A_{>k}, A_{<k}\), we take the approach in \cite{tang2025wavelet}.
We assume that all sketch tensors \(\{S_{>1}\} \cup \{S_{>j}, S_{<j}\}_{j=2}^{d-1} \cup \{S_{<d}\}\) are defined.
To obtain \(A_{>k}\), we construct a tensor \(Z_{>k}\) from contracting \(P\) with \(S_{<(k+1)} \otimes S_{>k}\).
The construction of the sketch tensors in \Cref{eqn: def of S_k in VI} implies the following equation for \(Z_k\):
\begin{equation}\label{eqn: Z_k_v interpolation formula}
  Z_{k}(\zeta_{k}, \omega_{k}) = P(v_{1}^{\zeta_{k}}, \ldots, v_{k-1}^{\zeta_{k}}, w_{k+1}^{\omega_{k}}, \ldots, w_{d}^{\omega_{k}}).
\end{equation}

One then obtains \(A_{<(k+1)}, A_{>k}\) through SVD of \(Z_{k}\), resulting in the best rank-\(r_k\) factorization \(Z_{k}(\zeta_{k}, \omega_{k}) = \sum_{\alpha_{k}}A_{<(k+1)}(\zeta_{k}, \alpha_{k})A_{>k}(\alpha_{k}, \omega_{k})\), which is typically done by truncated SVD.
One can see that obtaining \(A_{<k}\) is done through the best rank-\(r_{k-1}\) factorization of \(Z_{k-1}\).

\paragraph{The choice of sketch tensors}
When \(P\) is exactly a TT ansatz, one can achieve exact recovery as long as \(Z_{k}\) as defined in \Cref{eqn: Z_k_v interpolation formula} is of rank \(r_{k}\) for \(k = 1, \ldots, d-1\).
When \(P\) is well-approximated by a TT ansatz, the choice of the sketch tensors leads to different approximation errors, and an error analysis can be found in \cite{qin2022error}.
The TT-cross algorithm in \cite{oseledets2010tt} takes a particular choice of sketch tensor with a maximal volume heuristic, and we refer the readers to \cite{oseledets2010tt} for the details.

\subsubsection{TT-sketch algorithm for density estimation}
We cover here the procedure of the TT-sketch for the setting of density estimation.
The input to the algorithm is a collection of \(N\) samples \(\left(y_{1}^{(j)}, \ldots, y_{d}^{(j)}\right)_{j=1}^{N} \subset [n]^d\) distributed according to \(P\).
The output is a tensor train approximation for \(P\).

\paragraph{Obtain \(B_{k}\) by observable estimations}
In the TT-sketch algorithm, the tensors \(S_{>k}\) and \(S_{<k}\) are chosen so that their entry-wise evaluations are efficient.
We see that obtaining \(B_k\) in such cases reduces to repeated evaluations on the sketch tensors.
The construction of \(B_k\) implies the following equation for \(B_k\):
\begin{equation}\label{eqn: def of B_k in DE}
  \begin{aligned}
    B_{k}(\zeta_{k-1}, i_{k}, \omega_{k}) & = \sum_{i_{[k-1]}, i_{[d] - [k]}}P(i_{[d]})S_{<k}(\zeta_{k-1}, i_{[k-1]})S_{>k}(i_{[d] - [k]}, \omega_{k}) \\
                                          & \approx
    \frac{1}{N}\sum_{j = 1}^{N}S_{<k}(\zeta_{k-1}, y^{(j)}_{1}, \ldots, y^{(j)}_{k-1})\mathbf{1}(y^{(j)}_{k} = i_k)S_{>k}(y^{(j)}_{k+1}, \ldots, y^{(j)}_{d}, \omega_{k}),
  \end{aligned}
\end{equation}
where \(\mathbf{1}(z = i_k) = 1\) if and only if \(z = i_k\). When \(k = 1\) and \(k = d\), the construction of \(B_{1}\) and \(B_{d}\) are approximated by
\begin{equation}\label{eqn: def of B_1 in DE boundary}
  B_{1}(i_{1}, \omega_{1}) \approx \frac{1}{N}\sum_{j = 1}^{N}\mathbf{1}(y^{(j)}_{1} = i_1)S_{>1}(y^{(j)}_{k+1}, \ldots, y^{(j)}_{d}, \omega_{1}),
\end{equation}
and
\begin{equation}\label{eqn: def of B_d in DE boundary}
  B_{d}(\zeta_{d-1}, i_{d}) \approx \frac{1}{N}\sum_{j = 1}^{N}S_{<d}(\zeta_{d-1}, y^{(j)}_{1}, \ldots, y^{(j)}_{d-1})\mathbf{1}(y^{(j)}_{d} = i_d).
\end{equation}

\paragraph{Obtain \(A_k\) by observable estimations + SVD}
The construction of \(A_{>k}, A_{<k}\) is identical to the TT-cross case by constructing \(Z_{k}\), which is obtained by contracting \(P\) with \(S_{<(k+1)} \otimes S_{>k}\).
The formula for \(Z_k\) is:
\begin{equation}\label{eqn: Z_k DE formula}
  Z_{k}(\zeta_{k}, \omega_{k}) \approx \frac{1}{N}\sum_{j = 1}^{N}S_{<(k+1)}(\zeta_{k}, y^{(j)}_{1}, \ldots, y^{(j)}_{k})S_{>k}(y^{(j)}_{k+1}, \ldots, y^{(j)}_{d}, \omega_{k}).
\end{equation}

One then obtains \(A_{>k}\) through the best rank-\(r_k\) factorization of \(Z_{k}\) in the same way as in the TT-cross case, and \(A_{<k}\) is likewise obtained through the best rank-\(r_{k-1}\) factorization of \(Z_{k-1}\).

\paragraph{Error analysis of TT-sketch}
When \(P\) admits a TT ansatz and \(N \to \infty\), the TT-sketch algorithm converges to \(P\).
In \cite{tang2023generative}, it is proven that the approximation error \(\varepsilon\) goes to zero at an asymptotic rate of \(\varepsilon = O(\frac{d}{\sqrt{N}})\).
In TT-sketch, a good choice of sketch tensors leads to significant practical improvement in the approximation.
The design of sketch tensors in the TT-sketch case is currently ad hoc for general problems, but choosing a large number of sketch tensors is typically sufficient to ensure good practical performance.

\subsubsection{Summary of first stage}

We summarize the first stage of NTT compression in \Cref{alg: first stage}, which contains the fully specified implementation of TT-cross and TT-sketch.
Both methods are quite efficient in compressing \(P\) as a TT ansatz.
One can see that NTT is a subclass of TT, and therefore, the representation power of TT is stronger than that of NTT.
Therefore, the TT compression task should be successful as long as \(P\) can be well-represented by an NTT ansatz.
Conversely, failure in the first stage is a strong indication that one should use another ansatz to compress \(P\).

\begin{algorithm}
  \caption{Tensor train compression of distribution tensors}
  \label{alg: first stage}
  \begin{algorithmic}[1]
    \REQUIRE Sample \(\{y^{(i)}\}_{i = 1}^{N} \subset [n]^d\) distributed according to \(P\) or access to \(P \colon [n]^d \to \R\).
    \REQUIRE Collection of sketch tensors \(\{S_{>k}, S_{<k}\}\)
    \REQUIRE Target internal ranks $\{r_{k}\}$ for \(k = 1, \ldots, d-1\).
    \FOR{\(k = 1, \ldots, d-1\)}
    \STATE Obtain \(Z_{k}\) by \Cref{eqn: Z_k_v interpolation formula} or \Cref{eqn: Z_k DE formula}.
    \STATE Obtain \(A_{<(k+1)}\) and \(A_{>k}\) respectively as the left and right factor of the best rank \(r_{k}\) factorization of \(Z_{k}\).
    \ENDFOR
    \FOR{\(k = 1, \ldots, d\)}
    \IF{\(k = 1\)}
    \STATE Obtain \(B_{1}\) by \Cref{eqn: def of B_k in VI boundary} or \Cref{eqn: def of B_1 in DE boundary}.
    \STATE Obtain \(F_1\) by solving the linear equation in \Cref{eqn: low-rank boundary case}.
    \ENDIF
    \IF{\(k = d\)}
    \STATE Obtain \(B_{d}\) by \Cref{eqn: def of B_k in VI boundary} or \Cref{eqn: def of B_d in DE boundary}.
    \STATE Obtain \(F_d\) by solving the linear equation in \Cref{eqn: low-rank boundary case}.
    \ENDIF
    \IF{\(k \not \in \{1, d\}\)}
    \STATE Obtain \(B_{k}\) by \Cref{eqn: def of B_k in VI} or \Cref{eqn: def of B_k in DE}.
    \STATE Obtain \(F_k\) by solving the linear equation in \Cref{eqn: sketched linear system for F}.
    \ENDIF
    \ENDFOR
    \STATE Output the tensor train \(\tilde{P}\) with tensor component \((F_{k})_{k = 1}^{d}\).
  \end{algorithmic}
\end{algorithm}

\subsection{Stage two: non-negative tensor train fitting}\label{sec: NTT fitting}
We go through the second stage where one fits a non-negative tensor train ansatz \(P_{\G}\) against the reference tensor train \(\tilde{P}\).
Throughout this subsection, the reference TT ansatz \(\tilde{P}\) is defined with tensor component \((F_{k})_{k = 1}^{d}\), and the parameters for the NTT ansatz is the tensor component \(\G = (G_{k})_{k = 1}^{d}\).
We use \((\a_{k})_{k = 1}^{d-1}\) and \((r_{k})_{k = 1}^{d-1}\) to respectively denote the internal rank of \(P_{\G}\) and \(\tilde{P}\).
The tensor components \((F_{k})_{k = 1}^{d}\) are fixed with \(F_{1} \in \R^{ n \times r_{1}}\), \(F_{i} \in \R^{r_{i-1} \times n \times r_{i}}\) for \(i = 2,\ldots, d-1\), and \(F_{d} \in \R^{r_{d-1} \times n }\).
The tensor components \((G_{k})_{k = 1}^{d}\) are the optimization variables with \(G_{1} \in \R_{\geq 0}^{ n \times \a_{1}}\), \(G_{i} \in \R_{\geq 0}^{\a_{i-1} \times n \times \a_{i}}\) for \(i = 2,\ldots, d-1\), and \(G_{d} \in \R_{\geq 0}^{\a_{d-1} \times n }\).
While the internal ranks are not necessarily coupled, we typically choose \(\a_{k} \geq r_{k}\) to account for the fact that NTT has less representational power than TT.

\paragraph{Loss function in NTT fitting}
The NTT fitting procedure takes a variational perspective, where the loss function is defined as the squared Frobenius norm loss, along with log barrier functions defined on each tensor component.
The loss function is
\begin{equation}\label{eqn: NTT loss}
  \f(\G) = \lVert P_{\G} - \tilde{P} \rVert_{F}^2 + \sum_{i = 1}^{d}\mu_k \f_{k}(G_{k}),
\end{equation}
where \(\mu_{k}\) is the regularization parameter on \(\f_{k}\), and \(\f_{k}\) is the log barrier function placed on \(G_{k}\).
For \(k = 2, \ldots, d - 1\), one has
\[
  \f_{k}(G_k) = -\sum_{\alpha_{k-1}, i_k, \alpha_{k}} \log\left(G_k(\alpha_{k-1}, i_k, \alpha_k)\right),
\]
and \(\f_{k}\) for the boundary case of \(k = 1\) and \(k = d\) are respectively defined by
\[
  \f_1 = -\sum_{i_{1}, \alpha_{1}} \log\left(G_1(i_1, \alpha_{1})\right), \quad \f_{d} =-\sum_{\alpha_{d-1}, i_{d}} \log\left(G_d(\alpha_{d-1}, i_{d})\right).
\]
Moreover, for future reference, we write \(\f_{0}(\G) := \lVert P_{\G} - \tilde{P} \rVert_{F}^2\).
The loss function \(\f\) is thus defined by \(\f(\G) = \f_0(\G) + \sum_{k=1}^{d}\mu_k\f_k(G_k)\).

We remark that the non-negativity requirement for each \(G_k\) is automatically satisfied with the log barrier function \(\f_{k}\), and so the optimization task over \(\f\) is an unconstrained minimization problem.
Optimization with this loss function is essentially an implementation of the interior point method \cite{nocedal1999numerical}, with the log function chosen as the barrier function.

\paragraph{Single tensor component update}
We optimize the loss function \(\f\) with alternating minimization.
In other words, throughout the optimization procedure, we select one node \(G_k\) and optimize over \(G_k\) while keeping the other components fixed.
We describe the procedure for updating a single tensor component, \(G_k\).
The function \(\f_0\) can be written as
\[
  \f_{0}(\G) = \left<\tilde{P}, \tilde{P}\right> - 2\left<\tilde{P}, P_{\G}\right> + \left<P_{\G}, P_{\G}\right>,
\]
which is a quadratic function in \(G_k\).
Moreover, \(\f_k\) is by construction convex in \(G_{k}\).
Therefore, we use the Newton step for the alternating minimization.
Therefore, for updating each \(G_{k}\), we compute the update direction by taking \(\delta G_{k} = -(\nabla^{2}_{G_k} \f)^{-1}(\nabla_{G_k} \f)\).
One then performs an update \(G_k \gets G_{k} + \tau \delta G_{k}\), where \(\tau \in (0, 1)\) is chosen through a backtracking line search procedure.
For simplicity, we only perform one Newton step when updating each tensor component \(G_{k}\).

\paragraph{Optimization procedure}
We summarize the overall NTT fitting procedure in \Cref{alg: alternating minimization}.
The fitting is done by performing a series of single tensor component updates on \(G_k\) by iterating over \(k \in [d]\).
Moreover, we gradually decrease the coefficients on the log barrier by decreasing each \(\mu_k\).
After each sweep over \(k \in [d]\), we decrease the value of \((\mu_{k})_{k = 1}^{d}\).
The regularization schedule \(\{\mu_{k}^{i}\}_{i = 1}^{L}\) in \Cref{alg: alternating minimization} can be fixed as shown in \Cref{alg: alternating minimization}, but one can also choose \(\mu_{k}^{i}\) dynamically.
Lastly, the node schedule in \Cref{alg: alternating minimization} performs a forward sweep from \(1\) to \(d\) and then a backward sweep from \(d\) to \(1\).
The reason for the node update schedule is to cache and reuse intermediate tensor contractions for optimal efficiency.
As a result, running \Cref{alg: alternating minimization} has only a time complexity of \(O(dL)\).

\begin{algorithm}
  \caption{Non-negative tensor train fitting.}
  \label{alg: alternating minimization}
  \begin{algorithmic}[1]
    \REQUIRE Reference tensor train ansatz \(\tilde{P}\).
    \REQUIRE Total iteration number \(L\).
    \REQUIRE Regularization parameter schedule \(\{\mu_{k}^{i}\}_{i = 1}^{L}\) for \(k = 1, \ldots, d\).
    \REQUIRE Initial non-negative tensor train ansatz \(P_{\G}\) with tensor component \(G = (G_k)_{k = 1}^{d}\)
    \FOR{\(i = 1, \ldots, L\)}
    \STATE  \(\mu_{k} = \mu_{k}^{i}\) for \(k = 1,\ldots, d\).
    \STATE  \(\f = \f_0(\G) + \sum_{k=1}^{d}\mu_k\f_k(G_k)\).
    \FOR{\(k = 1,\ldots, d\) \textbf{and then} \(k = d,\ldots,1\)}
    \STATE  \(\Delta G_{k} = -(\nabla^{2}_{G_k} \f)^{-1}(\nabla_{G_k} \f)\).
    \STATE \(\tau \gets \text{Line\_search}(\f, G_k, \Delta G_k) \).
    \STATE \(G_k \gets G_k + \tau  \Delta G_k\).
    \ENDFOR
    \ENDFOR

    \STATE Output the tensor train \(P_{\G}\) with tensor component \((G_{k})_{k = 1}^{d}\).
  \end{algorithmic}
\end{algorithm}

\paragraph{Choice of \(\mu_k\)}
The choice of \(\mu_k\) is delicate in the NTT fitting setting.
For example, an NTT ansatz \(\G = (G_{k})_{k =1}^{d}\) represents the same tensor if \((G_{1}, G_{2})\) is replaced by \((cG_{1}, \frac{1}{c}G_{2})\).
Therefore, a good choice of \(\mu_{k}\) depends on the scale of \(G_{k}\).
Likewise, to ensure successful training, \(\mu_{k}\) needs to be tuned according to the scale of \(\f_0\).
However, as \(\lVert \tilde{P} \rVert_{F}\) might be exponentially large or small, the scale of \(\f_0\) varies dramatically in practice.

The aforementioned scaling issues can be addressed by the adaptive regularization strategy to be introduced in \Cref{sec: NTT fitting extra}.
However, one can also use a simple fixed regularization schedule after applying appropriate normalizations to the stated fitting problem.
First, we apply a normalization on \(\tilde{P}\).
By scaling \(\tilde{P}\) multiplicatively, one can ensure that \(\lVert \tilde{P} \rVert_{F} = 1\).
Second, for the initialization \(\G = (G_{k})_{k = 1}^{d}\) for \(P_{\G}\), one can apply multiplicative update on \(G_{k}\) such that \(\lVert P_{\G} \rVert = 1\) and that \(\lVert G_{k} \rVert_{F}\) is of roughly the same scale for all \(G_{k}\).
After the normalization is applied, one sees that the loss function \(\f_0\) is roughly \(O(1)\), and each tensor component in \(\G\) is of roughly the same magnitude.

In the numerical experiments in \Cref{sec: numerics}, we use a simple fixed regularization schedule.
In the notation of \Cref{alg: alternating minimization}, for each \(k = 1, \ldots, d\), the regularization schedule takes \(\mu_{k}^{1} = 10^{-3}\) and it takes \(\mu_{k}^{i + 1} = \max(1/2\mu_{k}^{i}, 10^{-12})\).
Therefore, we take an initial value of \(\mu_k = 10^{-3}\) and we decrease \(\mu_k\) by half until \(\mu_{k}\) reaches \(10^{-12}\).
We note that \(\f_0\) is strongly convex in \(G_k\), and Lemma 3.13 in \cite{forsgren2002interior} implies that the optimal \(G_k\) under the log barrier differs from the optimal unregularized \(G_k\) by an \(O(\mu_k)\) term.

\subsubsection{Implementation details in NTT fitting}\label{sec: implementation details}
We explain how the terms in the Newton step of \Cref{alg: alternating minimization} are calculated through tensor contractions.
The calculations are technical and can be skipped for the first read.
One sees that \(\nabla_{G_k}\f_{k}\) and \(\nabla^2_{G_k}\f_{k}\) are easy to calculate.
Thus, to perform the Newton step, it suffices to detail how to calculate \(\nabla_{G_k}\f_{0}\) and \(\nabla^2_{G_k}\f_{0}\).
One can check that \(\nabla_{G_k}\f_{0} = \nabla_{G_k}\left<P_{\G}, P_{\G}\right> - 2 \nabla_{G_k}\left<P_{\G}, \tilde{P}\right>\) and \(\nabla^2_{G_k}\f_{0} = \nabla^2_{G_k}\left<P_{\G}, P_{\G}\right>\).
Therefore, we shall go through how one calculates the gradient and hessian term for \(\left<P_{\G}, P_{\G}\right>\) and the gradient term for \(\left<P_{\G}, \tilde{P}\right>\).

We first calculate the gradient and hessian with respect to \(\left<P_{\G}, P_{\G}\right>\).
To do so, we define two tensors \(P_{\G, >k}\colon [\a_{k}]\times [n]^{d-k} \to \R_{\geq 0} \), \(P_{\G, <k} \colon [n]^{k-1} \times [\a_{k-1}] \to \R_{\geq 0}\) by the following equations:
\[
  P_{\G, >k}(\beta_{k}, i_{[d] - [k]}) = \sum_{\beta_{k+1}, \ldots, \beta_{d-1}}G_{k+1}(\beta_{k}, i_{k+1}, \beta_{k+1})\cdots G_{d}(\beta_{d-1}, i_{d}),
\]
and
\[
  P_{\G, <k}(i_{[k-1]}, \beta_{k-1}) = \sum_{\beta_{1}, \ldots, \beta_{k-2}}G_{1}(i_{1}, \beta_{1})\cdots G_{k-1}(\beta_{k-2}, i_{k-1}, \beta_{k-1}).
\]

The tensors are defined so that the following linear equation holds:
\[
  P_{\G}(i_1, \ldots, i_d) = \sum_{\beta_{k-1}, \beta_{k}}P_{\G, <k}(i_{[k-1]}, \beta_{k-1})G_{k}(\beta_{k-1}, i_k, \beta_{k})P_{\G, >k}(\beta_{k}, i_{[d] - [k]}).
\]

We then define two matrices \(M_{>k}\) and \(M_{<k}\) as follows:
\[
  M_{>k}(\beta_{k}, \beta_{k}') = \sum_{i_{[d] - [k]}}P_{\G, >k}(\beta_{k}, i_{[d] - [k]})P_{\G, >k}(\beta_{k}', i_{[d] - [k]}),
\]
and
\[
  M_{<k}(\beta_{k-1}, \beta_{k-1}') = \sum_{i_{[k-1]}}P_{\G, <k}(i_{[k-1]}, \beta_{k-1})P_{\G, <k}(i_{[k-1]}, \beta_{k-1}').
\]
We define \(M_{k} = M_{<k} \otimes M_{>k}\).
One can see that \(M_{k}\) does not depend on \(G_{k}\).
This construction allows one to calculate \(\nabla_{G_k}\left<P_{\G}, P_{\G}\right>\) and \(\nabla^2_{G_k}\left<P_{\G}, P_{\G}\right>\).
By construction, one has
\[
  \begin{aligned}
    \left<P_{\G}, P_{\G}\right> & = \sum_{i_{[d]}}
    P_{\G}(i_{[d]})P_{\G}(i_{[d]})                 \\  & = \sum_{i_{k},\beta_{k-1}, \beta_{k-1}', \beta_{k}, \beta_{k}'}G_{k}(\beta_{k-1}, i_k, \beta_{k})M_{k}((\beta_{k-1}, \beta_{k}), (\beta_{k-1}', \beta_{k}'))
       G_{k}(\beta_{k-1}', i_k, \beta_{k}').
  \end{aligned}
\]

Therefore, one can obtain the gradient term \(\nabla_{G_k}\left<P_{\G}, P_{\G}\right> \colon [\a_{k-1}] \times [n] \times [\a_{k}] \to \R\) as
\[
  \begin{aligned}
    \nabla_{G_k}\left<P_{\G}, P_{\G}\right>(\beta_{k-1}, i_k, \beta_{k})  = 2\sum_{\beta_{k-1}', \beta_{k}'}
    M_{k}((\beta_{k-1}, \beta_{k}), (\beta_{k-1}', \beta_{k}'))
    G_{k}(\beta_{k-1}', i_k, \beta_{k}').
  \end{aligned}
\]
Similarly, one can obtain the hessian term \(\nabla^2_{G_k}\left<P_{\G}, P_{\G}\right>\colon [\a_{k-1}]^2 \times [n]^2 \times [\a_{k}]^2 \to \R\) as
\[
  \nabla^2_{G_k}\left<P_{\G}, P_{\G}\right> = M_{<k} \otimes I_{n} \otimes M_{>k}.
\]

It remains to calculate the gradient term with respect to \(\left<P_{\G}, \tilde{P}\right>\).
Likewise, we define \(\tilde{P}_{>k}\colon [r_{k}]\times [n]^{d-k} \to \R \), \(\tilde{P}_{<k} \colon [n]^{k-1} \times [r_{k-1}] \to \R\) by the following equations:
\[
  \tilde{P}_{>k}(\alpha_{k}, i_{[d] - [k]}) = \sum_{\alpha_{k+1}, \ldots, \alpha_{d-1}}F_{k+1}(\alpha_{k}, i_{k+1}, \alpha_{k+1})\cdots F_{d}(\alpha_{d-1}, i_{d}),
\]
and
\[
  \tilde{P}_{<k}(i_{[k-1]}, \alpha_{k-1}) = \sum_{\alpha_{1}, \ldots, \alpha_{k-2}}F_{1}(i_{1}, \alpha_{1})\cdots F_{k-1}(\alpha_{k-2}, i_{k-1}, \alpha_{k-1}).
\]
Thus one obtains
\[
  \tilde{P}(i_1, \ldots, i_d) = \sum_{\alpha_{k-1}, \alpha_{k}}\tilde{P}_{<k}(i_{[k-1]}, \alpha_{k-1})F_{k}(\alpha_{k-1}, i_k, \alpha_{k})\tilde{P}_{>k}(\alpha_{k}, i_{[d] - [k]}).
\]

We define two matrices \(L_{>k}\) and \(L_{<k}\) as follows:
\[
  L_{>k}(\beta_{k}, \alpha_{k}) = \sum_{i_{[d] - [k]}}P_{\G, >k}(\beta_{k}, i_{[d] - [k]})\tilde{P}_{>k}(\alpha_{k}, i_{[d] - [k]}),
\]
and
\[
  L_{<k}(\beta_{k-1}, \alpha_{k-1}) = \sum_{i_{[k-1]}}P_{\G, <k}(i_{[k-1]}, \beta_{k-1})\tilde{P}_{<k}(i_{[k-1]}, \alpha_{k-1}).
\]

We define \(L_{k} = L_{<k} \otimes L_{>k}\).
By construction, one has
\[
  \begin{aligned}
    \left<P_{\G}, \tilde{P}\right> & = \sum_{i_{[d]}}
    P_{\G}(i_{[d]})\tilde{P}(i_{[d]})                 \\  & = \sum_{i_{k},\beta_{k-1}, \alpha_{k-1}, \beta_{k}, \alpha_{k}}G_{k}(\beta_{k-1}, i_k, \beta_{k})L_{k}((\beta_{k-1}, \beta_{k}), (\alpha_{k-1}, \alpha_{k}))
       F_{k}(\alpha_{k-1}, i_k, \alpha_{k}).
  \end{aligned}
\]
Therefore, the gradient term \(\nabla_{G_k}\left<P_{\G}, \tilde{P}\right> \colon [\a_{k-1}] \times [n] \times [\a_{k}] \to \R\) has the following formula:
\[
  \begin{aligned}
    \nabla_{G_k}\left<P_{\G}, \tilde{P}\right>(\beta_{k-1}, i_k, \beta_{k}) = \sum_{\alpha_{k-1}, \alpha_{k}}
    L_{k}((\beta_{k-1}, \beta_{k}), (\alpha_{k-1}, \alpha_{k}))
    F_{k}(\alpha_{k-1}, i_k, \alpha_{k}).
  \end{aligned}
\]

From the derived formulas, we show that each loop over \(k \in [d]\) in \Cref{alg: alternating minimization} has a cost of \(O(d)\).
The linear complexity is achieved by reusing intermediate results from tensor contractions.
For simplicity, we consider a forward sweep where one iterates \(j\) from \(1\) to \(d\).
For \Cref{alg: alternating minimization}, we have shown that performing an update on \(G_{j}\) requires \((M_{>j}, M_{<j})\) and \((L_{>j}, L_{<j})\).
Before any of the \(G_{j}\) component is updated, we pre-compute and store \((M_{>j}, M_{<j})\) and \((L_{>j}, L_{<j})\) for all \(j \in [d]\).
The pre-computation has an \(O(d)\) cost.
Then, one first updates \(G_{1}\).
Subsequently, one can use the updated \(G_{1}\) to recompute \(M_{<2}, L_{<2}\), which allows one to update \(G_{2}\), as \(M_{>2}, L_{>2}\) are still up-to-date.
In general, after updating \(G_{k}\), one can use the updated \(G_{k}\) to recompute \(M_{<(k+1)}, L_{<(k+1)}\).
One can see that \(M_{>(k+1)}\) and \(L_{>(k+1)}\) are up-to-date, and so one can update \(G_{k+1}\) and recompute \(M_{<(k+2)}, L_{<(k+2)}\).
This procedure continues until all tensor components \(G_1, \ldots, G_d\) are updated.
The cost to recompute \(M_{<(k+1)}, L_{<(k+1)}\) requires a tensor contraction involving \(M_{<k}, L_{<k}, F_{k}, G_{k}\) and is of \(O(1)\) cost.
Therefore, for \Cref{alg: alternating minimization}, each iteration over \(k \in [d]\) has a cost of \(O(d)\).

\subsection{Acceleration strategies in NTT fitting}\label{sec: NTT fitting extra}
This subsection discusses numerical techniques that accelerate the NTT fitting procedure in \Cref{alg: alternating minimization}.

\paragraph{Parallel parameter update}
We show that \(\nabla_{G_k}^2 \f\) has a block-diagonal structure.
As a consequence, obtaining the search direction \(\delta G_{k}=-(\nabla^{2}_{G_k} \f)^{-1}(\nabla_{G_k} \f)\) in \Cref{alg: alternating minimization} is of cost \(O(n)\) rather than the naive \(O(n^3)\) complexity.
We consider the general case where \(k \not \in \{1, d\}\).
We use \(\mathrm{vec}(\cdot)\) to denote the operation of flattening a tensor to a vector.
For \(j = 1, \ldots, n\), we define a vector-valued variable \(v_j \in \R^{\a_{k-1}\a_{k}}\) according to
\[
  v_j = \mathrm{vec}\left([G_k(\beta_{k-1}, j, \beta_{k})]_{\beta_{k-1} \in [\a_{k-1}], \beta_{k} \in [\a_{k}]}\right).
\]
In other words, \(v_j\) is the flattening of the slice of \(G_{k}\) obtained by fixing the second index to be \(i_k = j\).
To prove the \(O(n)\) complexity in the Newton step, it suffices to prove that \(\nabla_{v_{j}}\nabla_{v_{j'}}\f = 0\) whenever \(j \not = j'\).
Proving the statement shows that performing a single-node update on \(G_k\) reduces to performing an update on \(v_{j}\) for \(j = 1, \ldots, n\).
Moreover, the statement shows that the update on \(v_j\) can be done in parallel.

We give a proof for the desired statement.
One sees that \(\f_{k'}\) does not depend on \(G_{k}\) unless \(k = k'\).
Therefore, one has \(\nabla^2_{G_{k}}\f = \nabla^2_{G_{k}}\f_{0} + \mu_k\nabla^2_{G_{k}}\f_{k}\).
The function \(\f_k\) is a sum of univariate functions.
Therefore, \(\nabla_{\mathrm{vec}(G_k)}^2\f_k\) is a diagonal matrix and in particular \(\nabla_{v_{j}}\nabla_{v_{j'}}\f_k = 0\) unless \(j = j'\).
Thus, it remains to analyze \(\f_0\).
The only non-linear term in \(\f_{0}\) comes from \(\left<P_{\G}, P_{\G}\right>\).
By the derived formula in \Cref{sec: implementation details}, we see that there exist two matrices \(M_{>k}, M_{<k}\) such that \(
\left<P_{\G}, P_{\G}\right> = \sum_{j =1 }^{n} v^{\top}_{j} \left(M_{<k} \otimes M_{>k}\right)v_{j}\).
Therefore, for \(j \not = j'\), one has \(\nabla_{v_j}\nabla_{v_{j'}}\f = 0\) as is desired.

\paragraph{Newton step acceleration}
In practice, it suffices to obtain \(\delta G_{k}=-(\nabla^{2}_{G_k} \f)^{-1}(\nabla_{G_k} \f)\) approximately during the Newton step in \Cref{alg: alternating minimization}.
Therefore, we can further reduce the complexity of the Newton step by using the conjugate gradient (CG) method \cite{golub2013matrix}.
Utilizing the aforementioned parallel update strategy, one can divide \(G_k\) into \(n\) variables \(v_1, \ldots, v_n\), and a Newton step on \(G_k\) reduces to performing a Newton step on each variable in \((v_j)_{j = 1}^{n}\) in parallel.
For the general case where \(k \not \in \{1, d\}\), updating \(v_j\) with the Newton step is equivalent to solving the linear equation
\begin{equation}\label{eqn: CG system}
  \left(M_{<k} \otimes M_{>k} + \mu_k \mathrm{diag}(1 \oslash (v_j \odot v_j))\right)(\delta v_j) = - \nabla_{v_j} \f.
\end{equation}

We analyze the complexity of solving for \Cref{eqn: CG system}.
Let \(\a_{\mathrm{max}} = \max_{k \in [d-1]}\a_{k}\) be the maximal internal rank of \(P_{\G}\).
One can see that the solution \(\delta v_j\) is of dimension \(\a_{k-1}\a_k\) and so solving for \(\delta v_j\) in \Cref{eqn: CG system} naively has an \(O(\a_{\mathrm{max}}^6)\) cost, which is quite slow when \(\a_{\mathrm{max}}\) is large.
However, the coefficient matrix in \Cref{eqn: CG system} is special and its matrix-vector multiplication only has an \(O(\a_{\mathrm{max}}^3)\) cost.
As the goal is to approximately solve \Cref{eqn: CG system}, we terminate the CG algorithm after \(N_{\mathrm{iter}}\) steps, which leads to a complexity of \(O( N_{\mathrm{iter}} \a_{\mathrm{max}}^3)\).
When \(\mu_k\) is moderately big, one can see that the coefficient matrix in \Cref{eqn: CG system} is well-conditioned due to the diagonal term.
Lastly, we observe empirically that taking the diagonal term \(D = \mu_k \mathrm{diag}(1 \oslash (v_j \odot v_j))\) as the preconditioner gives a substantial speed-up to solving \Cref{eqn: CG system}.
Therefore, the numerical experiments in \Cref{sec: numerics} also include the performance for using the preconditioned conjugate gradient (PCG) method with \(D\) as the preconditioner.

\paragraph{Adaptive barrier update}
We introduce an adaptive barrier update schedule to dynamically choose \(\mu_k\). The goal is to reduce the log barrier coefficient \(\mu_k\) as rapidly as possible while ensuring that the current variable \(\G\) is approximately stationary in the regularized loss function \(\f(\G)\).
We apply a simple heuristic modified from \cite{nocedal2009adaptive}, which is commonly used for nonlinear interior methods.
We assume without loss of generality that \(k \not \in \{1, d\}\).
Using the notation from \Cref{alg: alternating minimization}, the adaptive regularization schedule is obtained by taking
\begin{equation}\label{eqn: mu schedule}
  \mu_{k}^{i+1} = \min\left(\mu_{k}^{i}, \tilde{\mu}_{k}^{i}\right), \quad \tilde{\mu}_{k}^{i} = \frac{\sigma}{\a_{k-1}n \a_k} \sum_{\beta_{k-1} = 1}^{\a_{k-1}}\sum_{\beta_{k} = 1}^{\a_{k}}\sum_{i_k = 1}^{n}G_{k}(\beta_{k-1}, i_k, \beta_{k}) \lvert \nabla_{G_{k}}\f_0(\beta_{k-1}, i_k, \beta_{k}))\rvert,
\end{equation}
where \(\sigma > 0\) is a centering parameter.
If one removes the absolute value in \Cref{eqn: mu schedule} for the formula for \(\tilde{\mu}_{k}^{i}\), one exactly recovers the heuristic barrier rule from \cite{nocedal2009adaptive}.
The value for \(\sigma\) can also be chosen through a heuristic, but this work considers a fixed value of \(\sigma\) for simplicity.
One can see that the formula in \Cref{eqn: mu schedule} leads to an adaptive and monotonically decreasing schedule.
When \(\f\) is stationary in terms of \(G_{k}\), the term \(\nabla_{G_k} \f_0\) is entry-wise positive.
Therefore, taking the absolute value in \Cref{eqn: mu schedule} is mainly used to increase robustness.
Lastly, the corresponding equation for \(k = 1\) and \(k = d\) can be obtained by respectively omitting the \(\beta_{k-1}\) and \(\beta_{k}\) index in \Cref{eqn: mu schedule}.

\paragraph{Warm initialization}
As is typical for second-order optimization, we employ warm initialization to reduce the number of Newton steps required to achieve convergence.
To achieve this, we employ a multiplicative update algorithm, as considered in \cite{shcherbakova2019nonnegative}.
We summarize the procedure in \Cref{alg: lee-seung}.
The symbols \(\odot\) and \(\oslash\) in \Cref{alg: lee-seung} respectively denote the entry-wise multiplication and division, and the maximization is taken entry-wise.
As in the Lee-Seung algorithm, the multiplicative update ensures that each tensor component stays non-negative.

\begin{algorithm}
  \caption{Multiplicative update algorithm for NTT fitting (from \cite{shcherbakova2019nonnegative}).}
  \label{alg: lee-seung}
  \begin{algorithmic}[1]
    \REQUIRE Reference tensor train ansatz \(\tilde{P}\).
    \REQUIRE Total iteration number \(N_{\mathrm{iter}}\).
    \REQUIRE Small positive number \(\lambda\) (e.g. \(\lambda = 10^{-9}\)).
    \REQUIRE Initial non-negative tensor train ansatz \(P_{\G}\) with tensor component \(\G = (G_k)_{k = 1}^{d}\)
    \FOR{\(i = 1, \ldots, N_{\mathrm{iter}}\)}
    \FOR{\(k = 1,\ldots, d\) \textbf{and then} \(k = d,\ldots,1\)}
    \STATE  \( U = \max(\nabla_{G_k} \left<\tilde{P}, P_{\G}\right>, \lambda) \)
    \STATE  \( V = \frac{1}{2}\nabla_{G_k} \left<P_{\G},  P_{\G}\right> \)
    \STATE \(G_k \gets G_k \odot U \oslash V\).
    \ENDFOR
    \ENDFOR
    \STATE Output the tensor train \(P_{\G}\) with tensor component \((G_{k})_{k = 1}^{d}\).
  \end{algorithmic}
\end{algorithm}

In our case, we use a random initialization on \(\G\) and run a few iterations of \Cref{alg: lee-seung} to obtain an initial parameter \(\G\).
Doing so ensures that \(P_{\G} \approx \tilde{P}\).
While the approximation error might still be large at this stage, the multiplicative update ensures that \(P_{\G}\) is of the same scale as \(\tilde{P}\) in terms of the squared Frobenius norm.
Then, we balance the magnitude of \(\G = (G_{k})_{k = 1}^{d}\) by replacing \((G_{k})_{k = 1}^{d}\) with \( (c_{k}G_{k})_{k = 1}^{d}\), where the rescaling parameter \(c_k\) is defined by the following equation \[\log(c_{k}) = -\log(\lVert G_{k} \rVert_{F}) + \frac{1}{d}\sum_{j = 1}^{d}\log(\lVert G_{j} \rVert_{F}).
\]
We see that \(\sum_{j=1}^{d}\log(c_{j}) = 0\) and so \(P_{\G}\) does not change by the rescaling.
The construction of \(c_k\) ensures \(\log(\lVert c_{k} G_{k} \rVert_{F}) = \frac{1}{d}\sum_{j = 1}^{d}\log(\lVert G_{j} \rVert_{F})\), and so the relative scale of each tensor component is identical after the rescaling.

To demonstrate the performance of \Cref{alg: alternating minimization} after the proposed acceleration strategy, we shall use \Cref{alg: lee-seung} as the benchmark in \Cref{sec: numerics} and compare the efficiency in NTT fitting.
The update step in \Cref{alg: lee-seung} can be derived from the Lee-Seung algorithm in \cite{lee2000algorithms}.
Essentially, the single node update in \Cref{alg: lee-seung} is equivalent to taking one step of the Lee-Seung algorithm in \cite{lee2000algorithms} by taking the squared Frobenius norm as the loss.
Therefore, \Cref{alg: lee-seung} does not consider any regularization and is an algorithm that directly minimizes \(\f_0\).
To better compare the numerical efficiency with \Cref{alg: alternating minimization}, we choose a symmetrized node update schedule in the version presented in \Cref{alg: lee-seung}.
In the original work in \cite{shcherbakova2019nonnegative}, the node schedule iterates \(k\) in increasing order from \(1\) to \(d\).

\section{Self-concordance analysis}\label{sec: self-concordance}
This section analyzes the \emph{single-component} optimization step that \Cref{alg: alternating minimization} performs inside the NTT fitting loop.
We show that the inner Newton step enjoys rapid quadratic convergence if one takes multiple Newton steps during a single component update.
Due to the analysis, one can see that the bottleneck in convergence in fitting error largely lies in the outer loop.
In practice, it suffices to take only a few Newton steps to guarantee rapid convergence. The version of NTT shown in \Cref{alg: alternating minimization} thus only performs one Newton step per node visited. The algorithm could also be modified to allow taking multiple Newton steps.

For a fixed outer iteration $i$ and node index $k$, the single-component subproblem is
\begin{align}
  \label{eq:single_opt}
  \min_{G_k}\;
  \f_\mu(G_k)
  \;=\;
  \f_0(\G)+\mu\,\f_k(G_k),
\end{align}
where \(\mu = \mu_k\), the main loss term is
$\f_0(\G)=\lVert P_\G-\widetilde P\rVert_F^{2}$,
and the barrier function is the entry-wise log barrier defined by
\begin{align*}
  \f_k(G_k) \;=\;
  -\!\!\sum_{\alpha_{k-1},\,i_k,\,\alpha_k}
  \log G_k(\alpha_{k-1},i_k,\alpha_k).
\end{align*}
Throughout this section, we treat the tensor component $G_k$ as a flattened vector of dimension $\nu_k=n\,a_{k-1}a_k$.

\subsection{Background on self-concordance}

\begin{definition}[Self-concordant function]
  \label{def:self_concordant}
  A convex function $F \colon \R^\nu \to \R$ is self-concordant if
  \begin{align*}
    \bigl|D^{3}
    F(x)[h,h,h]\bigr|
    \;\le\;
    2\,\bigl(D^{2}F(x)[h,h]\bigr)^{3/2},
    \qquad
    \forall\,x,\;h \in \R^\nu .
  \end{align*}
\end{definition}
Importantly, the univariate log barrier is self-concordant.
We provide a short proof below.

\paragraph{Log barrier is self-concordant}
For the univariate function $\phi(z)=-\log z$, one has
$\phi'''(z)=-2/z^{3}$ and $\phi''(z)=1/z^{2}$.
Hence, $\lvert\phi'''\rvert\le 2(\phi'')^{3/2}$, and so $\phi$ is self-concordant.
Sums preserve self-concordance (see discussion of self-concordant calculus in \cite[§9.6.2]{boyd2004convex}), so
\begin{align*}
  \Phi(x)= -\sum_{i=1}^{\nu_k}\log x_i
\end{align*}
is self-concordant on $\mathbb R_{>0}^{\nu_k}$.
As a result, the barrier $\f_k(G_k)$ in \Cref{eq:single_opt} is self-concordant on the positive orthant $\mathbb R_{>0}^{\nu_k}$.
Further, adding the quadratic term $\f_0(\G)$ does not affect self-concordance, so the full objective $\f_\mu(G_k)$ is also self-concordant on $\mathbb R_{>0}^{\nu_k}$ for all node indices $k$.

\subsection{Convergence guarantee during single-component updates}
Before we present our main theoretical results, we provide the details of the line search used in \Cref{alg: alternating minimization}, as well as introduce some notation related to the Newton step.

\paragraph{Backtracking line search}
We use the standard backtracking line search (Algorithm 9.2 in \cite{boyd2004convex}) in \Cref{alg: alternating minimization} as shown in \Cref{alg:backtracking}.

\begin{algorithm}[H]
  \caption{Backtracking line search for Newton step}
  \label{alg:backtracking}
  \begin{algorithmic}[1]
    \REQUIRE Newton direction $\Delta G_k$, $\alpha\in(0,0.5)$, $\beta\in(0,1)$
    \STATE $t\gets1$
    \WHILE{$\f_{\mu}(G_k+t\Delta G_k) > \f_{\mu}(G_k) + \alpha t\,\langle\nabla_{G_k}\f_{\mu},\Delta G_k\rangle$}
    \STATE $t\gets\beta t$
    \ENDWHILE
    \RETURN $t$
  \end{algorithmic}
\end{algorithm}

\Cref{alg:backtracking} multiplies $\Delta G_k$ by
$t \in(0,1]$, and the scaling constant \(t\) shrinks until the proposed point $G_k+t\Delta G_k$ is feasible and
the Armijo inequality
\[
  \f_\mu(G_k+t\Delta)\le
  \f_\mu(G_k)+\alpha\,t\,\langle\nabla \f_\mu,\Delta\rangle
\]
holds, guaranteeing descent.
With $\alpha\in(0,0.5)$ and $\beta\in(0,1)$ the loop always terminates; see \cite[§9.2]{boyd2004convex}.

\paragraph{Newton direction and decrement}
In the context of the Newton step, we introduce two important and useful quantities, which are the Newton direction and the Newton decrement.
\begin{definition}[Newton direction]
  \label{def:newton_direction}
  Let $G_k$ be a feasible point for \Cref{eq:single_opt}.
  The \emph{Newton direction} $\Delta G_k$ is defined as
  \begin{align*}
    \Delta G_k
     & = -\bigl[\nabla_{G_k}^{2}\f_\mu\bigr]^{-1}\nabla_{G_k}\f_\mu,
  \end{align*}
  where $\nabla_{G_k}^{2}\f_\mu$ is the Hessian of $\f_\mu$ at $G_k$ and
  $\nabla_{G_k}\f_\mu$ is the gradient.
\end{definition}

\begin{definition}[Newton decrement]
  \label{def:newton_decrement}
  Given a feasible point $G_k$ for \Cref{eq:single_opt}, the
  \emph{Newton decrement} $\lambda(G_k)$ is defined as
  \begin{align*}
    \lambda(G_k)
     & = \left(\nabla_{G_k}\f_\mu^{\mathsf T}\,
    \bigl[\nabla_{G_k}^{2}\f_\mu\bigr]^{-1}\,
    \nabla_{G_k}\f_\mu\right)^{1/2}.
  \end{align*}
\end{definition}

Intuitively, $\Delta G_k$ is the optimal descent step assuming a quadratic model, and
$\lambda(G_k)$ measures how far the point $G_k$ is from the minimizer of the quadratic model, in the norm induced by the local Hessian $\nabla_{G_k}^{2}\f_\mu$.

We now present the main theoretical results of this section, which are broadly analogous to the standard results on Newton steps for self-concordant barriers, but adapted to the context of the NTT single component optimization.

\begin{theorem}[Guarantees for the backtracked Newton step in NTT component optimization]
  \label{thm:newton}
  Let $G_k$ be feasible for \Cref{eq:single_opt}.
  Let $\Delta=\Delta G_k$ be the Newton direction. Let $\lambda=\lambda(G_k)$ and let
  $t$ be the step size returned by
  \Cref{alg:backtracking} with parameter \(\alpha, \beta\).
  Then
  \begin{enumerate}
    \item \textbf{Feasibility}: $G_k^{+}:=G_k+t\Delta$ is strictly positive entry-wise.
    \item \textbf{Descent}: $\f_\mu(G_k^{+})\le \f_\mu(G_k)-\alpha\,t\,\lambda^{2}$.
    \item \textbf{Quadratic convergence}:
          If $\lambda\le\frac14$, then $t=1$ is accepted in the backtracking line search, and we have quadratic convergence:
          \begin{align*}
            \lambda(G_k^{+}) & \le 2\lambda(G_k)^2,                  \\
            \f_\mu(G_k^{+})   & \le \f_\mu(G_k)-\tfrac12\,\lambda^{2}.
          \end{align*}
  \end{enumerate}
\end{theorem}

\begin{proof}
  \textit{Step 1 – Feasibility}\par
  We use $\lVert V\rVert_{G_k}^{2}=V^{\mathsf T}\nabla^{2}\f_\mu(G_k)V$ to denote the local norm induced by the Hessian at the current point $G_k$.
  A standard self-concordance result is the \emph{Hessian sandwich inequality} \cite{nesterov1994interior},
  valid whenever $\lVert V\rVert_{G_k}<1$:
  \begin{align*}
    \bigl(1-\lVert V\rVert_{G_k}\bigr)^{2}\,
    \nabla^{2}\f_\mu(G_k)
    \;\preceq\;
    \nabla^{2}\f_\mu(G_k+V)
    \;\preceq\;
    \frac{1}{\bigl(1-\lVert V\rVert_{G_k}\bigr)^{2}}\,
    \nabla^{2}\f_\mu(G_k).
  \end{align*}

  \medskip
  \noindent
  \emph{Quantifying the growth of the norm along the ray.}
  Set
  \begin{align*}
    m(t)
    = V^{\mathsf T}\,\nabla^{2}\f_\mu\bigl(G_k+tV\bigr)\,V,
    \qquad 0\le t<1 .
  \end{align*}
  Applying the right–hand side of the inequality gives
  \begin{align*}
    m(t)
     & \;\le\;
    \frac{1}{\bigl(1-t\lVert V\rVert_{G_k}\bigr)^{2}}
    \;V^{\mathsf T}\,\nabla^{2}\f_\mu(G_k)\,V \\[4pt]
     & =\frac{\lVert V\rVert_{G_k}^{2}}
    {(1-t\lVert V\rVert_{G_k})^{2}}          \\[4pt]
     & \le \frac{\lambda^{2}}
    {(1-t\lambda)^{2}}
    \qquad\text{for } 0\le t<1/\lambda.
  \end{align*}
  Taking the square root, we obtain an explicit bound on the
  \emph{local} norm at the new point:
  \begin{equation}\label{eqn: bound on V in local norm}
    \lVert V\rVert_{G_k+tV}
    \;\le\;
    \frac{\lambda}{1-t\lambda},
    \qquad 0\le t<\frac1\lambda,
  \end{equation}
  where the right-hand side is finite whenever $t\lambda<1$.

  \medskip
  \noindent
  \emph{Applying (F.1) to the Newton direction.}
  For the full Newton step we take $V=\Delta$ and by definition one has
  $\lVert\Delta\rVert_{G_k}=\lambda$.
  The back-tracking loop multiplies $t$ by $\beta$ repeatedly
  until $t\lambda<1$, so the accepted step $t\Delta$ satisfies the
  premise of \Cref{eqn: bound on V in local norm}.
  Consequently every local norm along the segment
  $G_k+\tau t\Delta$ ($0\le\tau\le1$) is finite, which in turn implies
  every coordinate of $G_k+\tau t\Delta$ stays \emph{strictly} positive
  (otherwise the logarithmic barrier would blow up to $+\infty$).
  Therefore
  \begin{align*}
    G_k^{+}=G_k+t\Delta
  \end{align*}
  remains feasible, completing the proof of Step 1.

  \medskip
  \textit{Step 2 – Descent}\par
  The backtracking line search accepts $t$ only if
  $
    \f_\mu(G_k^{+})
    \le \f_\mu(G_k)+\alpha\,t\,\langle\nabla \f_\mu,\Delta\rangle.
  $
  Because $\langle\nabla \f_\mu,\Delta\rangle=-\lambda^{2}$ by definition of the Newton direction,
  we obtain $\f_\mu(G_k^{+})\le \f_\mu(G_k)-\alpha t\lambda^{2}$.

  \medskip
  \textit{Step 3 – Quadratic regime ($\lambda\le\frac14$)}\par
  When $\lambda\le\frac14$, one has $\lVert\Delta\rVert_{G_k}<1$, and so taking the step size $t=1$ is feasible.
  Set $\varphi(t)=\f_\mu(G_k+t\Delta)$ for $t\in[0,1]$.
  Then $\varphi'(0)=-\lambda^{2}$ and, by self-concordance, one has
  $\varphi''(t)\le\lambda^{2}/(1+t\lambda)^{2}$.

  \medskip
  \noindent
  \emph{Integral-remainder version of Taylor’s theorem}:
  \begin{align*}
    \f_\mu(G_k+\Delta)-\f_\mu(G_k)
     & = \varphi(1)-\varphi(0) \\[4pt]
     & = \varphi'(0)
    +\int_{0}^{1}(1-t)\,\varphi''(t)\,dt .
  \end{align*}
  Insert $\varphi'(0)$ and the bound on $\varphi''(t)$:
  \begin{align*}
    \f_\mu(G_k+\Delta)-\f_\mu(G_k)
     & \le
    -\lambda^{2}
    +\lambda^{2}\!
    \int_{0}^{1}
    \frac{1-t}{(1+t\lambda)^{2}}\,dt .
  \end{align*}

  \medskip
  \noindent
  \emph{Evaluate the integral explicitly}.
  Substitute $u=1+t\lambda$ ($du=\lambda\,dt$):
  \begin{align*}
    \int_{0}^{1}\frac{1-t}{(1+t\lambda)^{2}}\,dt
     & =\int_{u=1}^{\,1+\lambda}
    \frac{1-\dfrac{u-1}{\lambda}}{u^{2}}\,
    \frac{du}{\lambda}            \\[6pt]
     & =\frac{1}{\lambda^{2}}
    \int_{1}^{1+\lambda}
    \frac{\lambda+1-u}{u^{2}}\,du \\[6pt]
     & =\frac{1}{\lambda^{2}}
    \Bigl[
    (\lambda+1)\underbrace{\int_{1}^{1+\lambda}\!u^{-2}\,du}_{=\,1-\frac{1}{1+\lambda}}
    \;-\;
    \underbrace{\int_{1}^{1+\lambda}\!u^{-1}\,du}_{=\,\log(1+\lambda)}
    \Bigr]                        \\[6pt]
     & =\frac{1}{\lambda^{2}}
    \Bigl[\lambda-\log(1+\lambda)\Bigr].
  \end{align*}

  \medskip
  \noindent
  \emph{Combine the pieces}:
  \begin{align*}
    \f_\mu(G_k+\Delta)-\f_\mu(G_k)
     & \le
    -\lambda^{2}
    +\bigl[\lambda-\log(1+\lambda)\bigr] .
  \end{align*}
  For $0\le\lambda\le\frac14$ we have the elementary bound
  $\lambda-\log(1+\lambda)\le\frac12\lambda^{2}$,
  giving
  \begin{align*}
    \f_\mu(G_k+\Delta)\le \f_\mu(G_k)-\tfrac12\lambda^{2}.
  \end{align*}
  To prove the quadratic convergence claim, we first define the following for brevity,
  \[
    H:=\nabla^2 \f_\mu(G_k),\qquad
    g:=\nabla \f_\mu(G_k),\qquad
    g_+ := \nabla \f_\mu(G_k^+).
  \]
  For $t\in[0,1)$ let $H_t=\nabla^2 \f_\mu(G_k+t\Delta)$.
  The Hessian sandwich takes the form
  \[
    (1-t\lambda)^2H \;\preceq\; H_t \;\preceq\; (1-t\lambda)^{-2}H.
  \]
  With $\displaystyle\delta\coloneqq\int_0^1 H_sH^{-1}\,ds$, the integral-remainder version Taylor’s theorem gives
  \[
    \nabla f(G_k^+) \;=\; g - \delta g \;=\; (I-\delta)g.
  \]

  \paragraph{Bounding $\delta$}
  From the Hessian sandwich,
  \[
    \int_0^1\!
    \!(1-t\lambda)^2\,dt\,I \;\preceq\; \delta \;\preceq\;
    \int_0^1\!\!(1-t\lambda)^{-2}\,dt\,I.
  \]
  Evaluating the two scalar integrals, one has
  \[
    k_1 := \int_0^1 (1-t\lambda)^2\,dt
    = 1-\lambda+\frac{\lambda^{2}}{3},\qquad
    k_2 := \int_0^1 (1-t\lambda)^{-2}\,dt
    = \frac{1}{1-\lambda},
  \]
  so that
  \[
    k_1 I \;\preceq\; \delta \;\preceq\; k_2 I.
  \]
  Hence for $\lambda\in[0,1/4]$ we have
  \[
    \|I-\delta\|_2
    \;\leq \max\{\,1-k_1,\;k_2-1\,\}
    \;=\; \frac{\lambda}{1-\lambda}.
  \]
  Consequently
  \[
    g_+^\top H^{-1} g_+
    \;=\; g^\top(I-\delta)^{\!\top}
    H^{-1}(I-\delta)g
    \;\le\; \frac{\lambda^{4}}{(1-\lambda)^{2}}.
  \]

  \paragraph{Norm change from $H$ to $H_1=\nabla^2 f(G_k^+)$}
  Taking $t=1$ in the Hessian sandwich gives
  \[
    (1-\lambda)^2H \;\preceq\; H_1 \;\preceq\; (1-\lambda)^{-2}H,
    \quad\Longrightarrow\quad
    H_1^{-1} \preceq (1-\lambda)^{-2}H^{-1}.
  \]

  \paragraph{Newton decrement at $G_k^+$}
  Combining the norm change with the bound on $g_+^\top H^{-1} g_+$ yields
  \begin{align*}
    \lambda(G_k^+)^{2}
     & = g_+^\top H_1^{-1} g_+                    \\
     & \le (1-\lambda)^{-2}\, g_+^\top H^{-1} g_+ \\
     & \le \frac{\lambda^{4}}{(1-\lambda)^{4}}.
  \end{align*}
  Taking square roots, we get
  \[
    \;
    \lambda(G_k^+)\;\le\;
    \dfrac{\lambda(G_k)^{2}}{\bigl(1-\lambda(G_k)\bigr)^{2}}.
    \!
  \]
  For $0 \le \lambda \le \tfrac14$, this simplifies to
  \[
    \lambda(G_k^+)\;\le\; 2\,\lambda(G_k)^{2},
  \]
  completing the proof of Step 3.

\end{proof}

\section{Numerical experiments}\label{sec: numerics}

We present the numerical performance of the proposed NTT compression procedure for variational inference and density estimation.
\Cref{sec: VI numerics} presents numerical experiments for the variational inference setting. \Cref{sec: DE numerics} presents the result for the density estimation setting. 

In particular, for the NTT fitting stage, we compare the different acceleration strategies discussed in \Cref{sec: NTT fitting extra} and use the multiplicative update algorithm considered in \cite{shcherbakova2019nonnegative} as the benchmark.
We briefly describe the specific parameter choices and details for each acceleration strategy below.
\begin{enumerate}
  \item \textbf{Fixed barrier schedule.}
        We use the regularization schedule discussed in \Cref{sec: NTT fitting} with an initial value of \(\mu_k = 10^{-3}\), and the value is halved every iteration until a minimal value of \(\mu_k = 10^{-12}\).
  \item \textbf{Adaptive barrier schedule.}
        We use the adaptive regularization schedule discussed in \Cref{sec: NTT fitting extra} with a centering parameter \(\sigma\).
  \item \textbf{Newton Acceleration with CG.}
        We use the conjugate gradient method to approximately solve the Newton step in \Cref{alg: alternating minimization}.
        We limit the maximum number of CG iterations for each step to \(N_{\mathrm{iter}} = 100\).
  \item \textbf{Preconditioned CG Newton.}
        We use the diagonal preconditioner described in \Cref{sec: NTT fitting extra} to accelerate the Newton steps.
        Since the diagonal weight $\mu$ monotonically decreases, we use the preconditioner until a threshold of \(\mu_k = 10^{-8}\), beyond which standard CG is used.
\end{enumerate}

For simple performance comparisons across all experiments for NTT fitting, we use the loss function \(\lVert P - P_{\G}\rVert^2_{F}/\lVert P\rVert^2_{F}\), which is the relative squared Frobenius norm error. We plot the convergence in wall-clock time. All experiments are run on a MacBook Pro with an M1 Pro processor.

\subsection{Variational inference}\label{sec: VI numerics}

\begin{figure}[ht]
  \centering
    \subcaptionbox{Example 1, Ginzburg-Landau model.\label{fig:numerics-gl}}[0.76\linewidth]{
    \includegraphics[width=\linewidth]{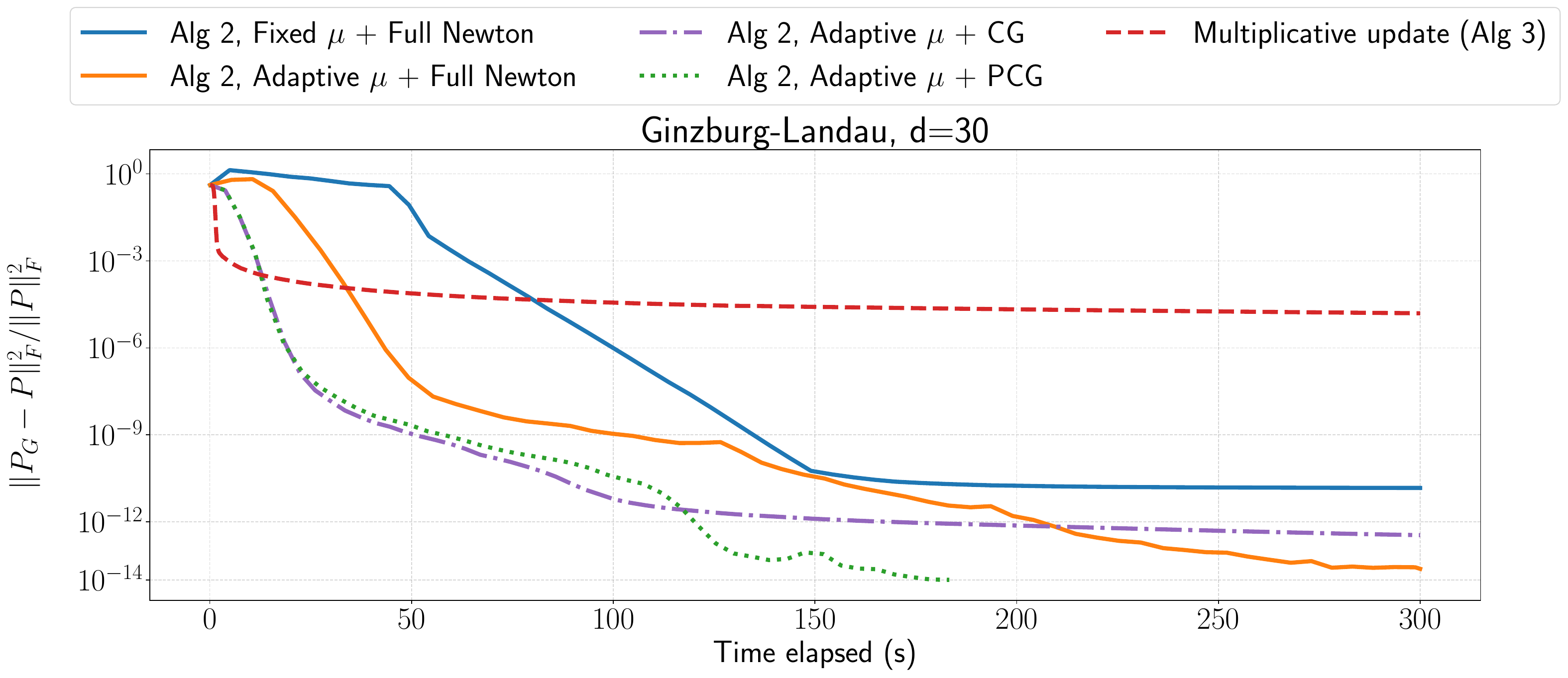}
  }
    \hfill
  \centering
  \subcaptionbox{Example 2, Gibbs kernel tensor.\label{fig:numerics-gibbs}}[0.75\linewidth]{
    \includegraphics[width=\linewidth]{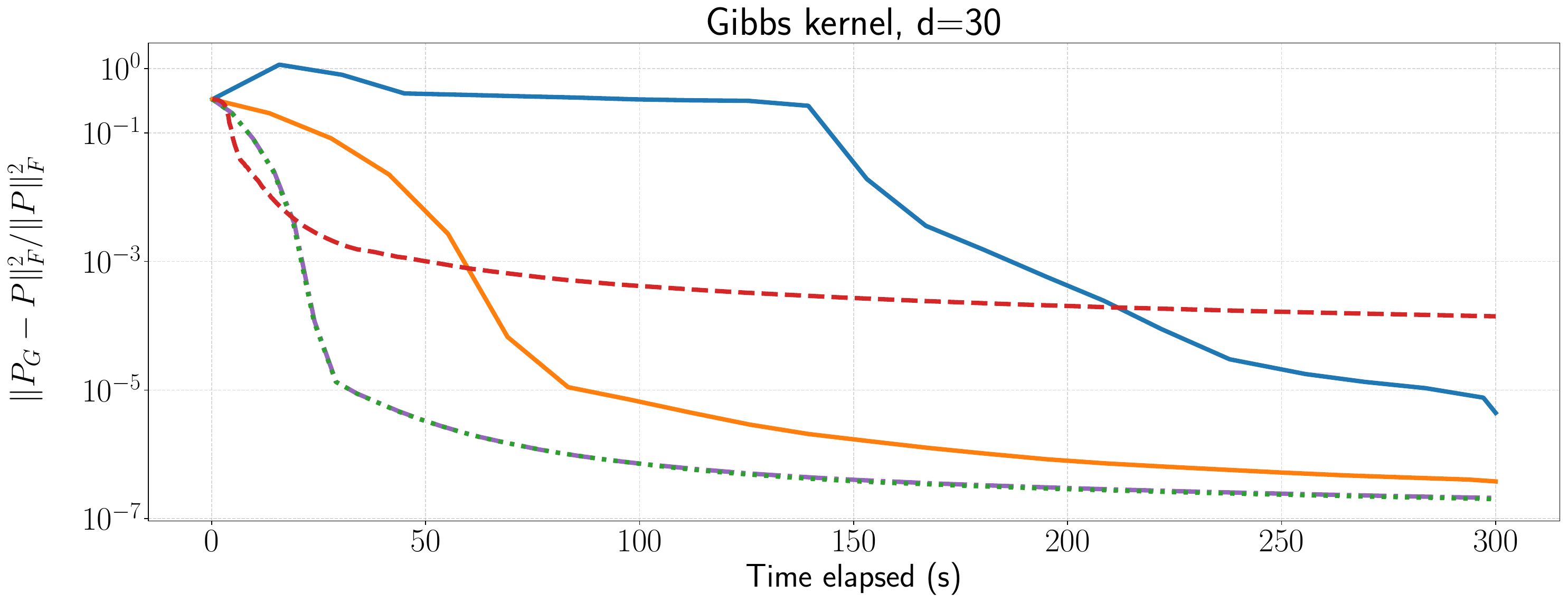}
  }
  \hfill
  \subcaptionbox{Example 3, Heavy-tail model.\label{fig:numerics-cauchy}}[0.75\linewidth]{
    \includegraphics[width=\linewidth]{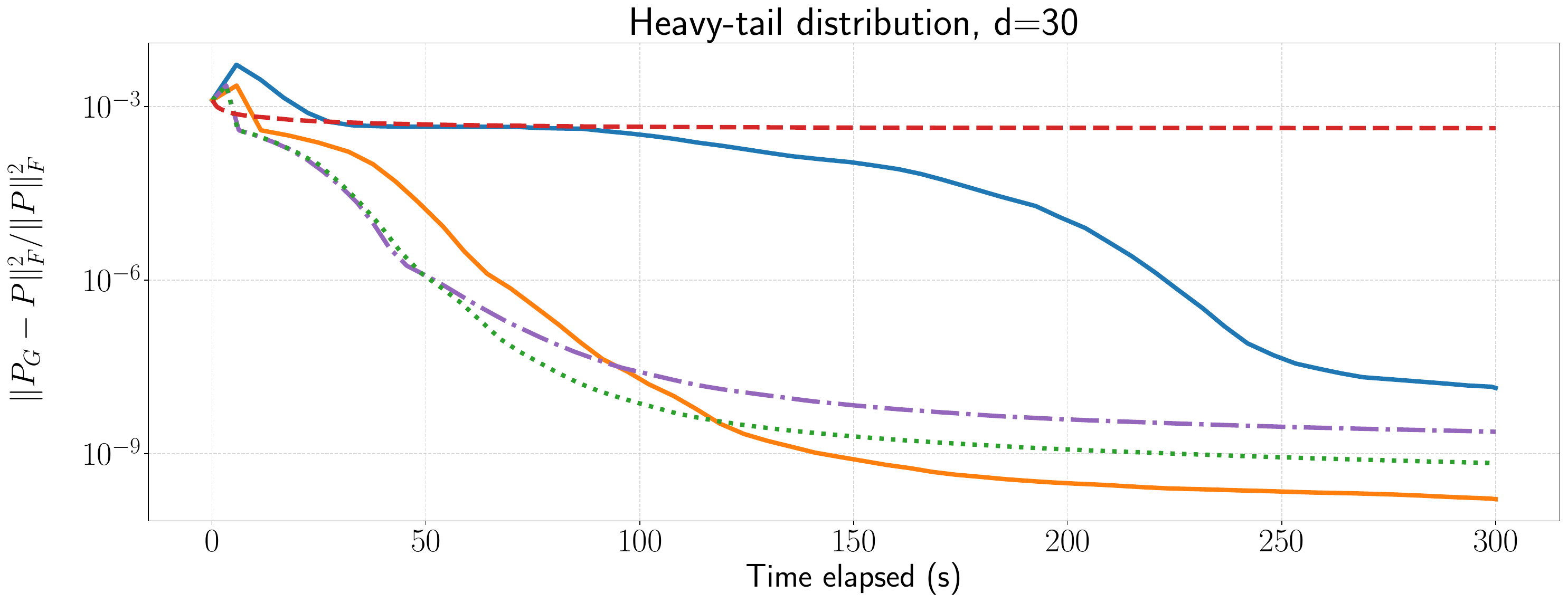}
  }

  \caption{Numerical results for NTT fitting in the variational inference setting. The plot compares the baseline multiplicative update (\Cref{alg: lee-seung}) with our main algorithm (\Cref{alg: alternating minimization}), accelerated using the different strategies described in \Cref{sec: NTT fitting extra}. The label CG and PCG corresponds to conjugate gradient and preconditioned conjugate gradient for the Newton step. For the centering parameter in the adaptive barrier strategy, we choose \(\sigma = 0.2\) for the first two examples, and for Example 3, we choose \(\sigma = 0.01\).
  }
  \label{fig:numerics}
\end{figure}

We test the NTT compression procedure for the variational inference setting. To test the accuracy in fitting the density \(P\), we randomly choose \(10^{5}\) entries from the total index set \([n]^d\), and we compare \(P\) against \(\tilde{P}, P_{\G}\) on those entries. In particular, we report the average relative loss between the evaluations on those entries. As the main contribution of this work is the NTT fitting procedure, our main goal is for \(P_{\G}\) to have a small error.

\paragraph{Example 1: Ginzburg-Landau model}
We let \(n = 50\) and we let \((x_i)_{i=1}^{n}\) be a uniform grid point on \([-2, 2]\) of length \(n\).
We first consider the unnormalized distribution tensor \(P \colon [n]^{d} \to \R\) defined by:
\begin{equation}\label{eqn: GL model}
    P(i_1, \ldots, i_d) = p(x_{i_1}, \ldots, x_{i_d}), \quad p(z_1, \ldots, z_d) = \exp\left(-\frac{\gamma}{2} \sum_{i = 1}^{d-1} \left(z_{i} - z_{i+1}\right)^2 - \frac{\lambda}{2} \sum_{i = 1}^{d} \left(1 - z_i^2\right)^2\right),
\end{equation}
where the dimension is \(d = 30\), and the Laplacian term and double-well term parameter is \(\gamma = \lambda = 0.16\). The formula for \(P\) in \Cref{eqn: GL model} is a discretized Ginzburg-Landau model.

We test the performance of the proposed two-stage approach on the model. In the first stage, we employ the TT-cross algorithm, utilizing the analytic formula of \(P\) as defined in \Cref{eqn: GL model}, with a maximal internal rank of \(r_{\max} = 10\).
The first stage outputs a TT approximation \(\tilde{P}\). The average relative entry-wise error on the randomly chosen \(10^{5}\) entries is \(1.6 \times 10^{-10}\), which shows that TT-cross is successful in compressing \(P\) with \(\tilde{P}\).

In the second stage, we perform the NTT fitting procedure on \(\tilde{P}\) using \Cref{alg: alternating minimization}.
We use a maximal rank of $\a = 20$ for the NTT ansatz $P_{\G}$.
The results for the Ginzburg-Landau model are shown in \Cref{fig:numerics-gl}.
We observe that \Cref{alg: alternating minimization} finds NTT solutions that achieve machine precision ($\approx 10^{-14}$) in the squared relative Frobenius error between the NTT $P_{\G}$ and the output of the first stage $\tilde{P}$.
Furthermore, applying the proposed acceleration strategy leads to a significant speed-up in the convergence of \Cref{alg: alternating minimization}, with the combination of an adaptive barrier schedule and preconditioned CG for Newton steps being the most efficient.

Finally, we report the performance of \(P_{\G}\) in terms of its evaluation on the \(10^{5}\) randomly chosen entries. We pick the NTT ansatz \(P_{\G}\) obtained using the adaptive barrier schedule and preconditioned CG. On those points, the average relative entry-wise error is \(3.9 \times 10^{-7}\), which shows that our procedure is successful in compressing \(P\) with an NTT ansatz.

\paragraph{Example 2: Gibbs kernel tensor}
We let \(n = 50\) and we let \((x_i)_{i=1}^{n}\) be a uniform grid point on \([-1, 1]\) of length \(n\).
We consider the unnormalized distribution tensor \(P \colon [n]^{d} \to \R\) defined as follows:
\begin{equation}\label{eqn: Gibbs kernel}
  P(i_1, \ldots, i_d) = p(x_{i_1}, \ldots, x_{i_d}), \quad p(z_1, \ldots, z_d) = \exp\left(-\beta\sum_{i=1}^{d-1}\lVert z_{i}-z_{i+1}\rVert\right),
\end{equation}
where the dimension is \(d = 30\) and the inverse temperature is \(\beta = 0.3\).
This setting is of particular practical importance as \(P\) is the Gibbs kernel tensor for the multi-marginal Sinkhorn algorithm \cite{marino2020optimal} when the transport cost is \(c(z_1, \ldots, z_d) = \sum_{i=1}^{d-1}\lVert z_{i}-z_{i+1}\rVert\).
Giving \(P\) an NTT ansatz is enormously beneficial if one were to perform the Sinkhorn scaling algorithm on the NTT ansatz instead of the full tensor. 

For this example, we consider a compression setting in which the goal is to approximate \(P\) with an NTT ansatz of a small rank, and we note that \(P\) admits an NTT ansatz if one takes the rank to be \(r = n\). For the first stage, we use the adaptive cross approximation method in \cite{savostyanov2011fast} to compress \(P\) as a tensor train ansatz of a maximal internal rank \(r_{\max} = 20\).
The average relative entry-wise error on the randomly chosen \(10^{5}\) entries is \(8.5 \times 10^{-3}\), which shows that the first stage is moderately successful in compressing \(P\) with \(\tilde{P}\), but the error is not negligible.

In the second stage, we use a maximal rank of $\a = 25$ for the NTT ansatz $P_{\G}$.
The numerical results for the Gibbs kernel tensor are shown in \Cref{fig:numerics-gibbs}.
We observe that \Cref{alg: alternating minimization} allows a fast decrease of the squared relative Frobenius error. The combination of an adaptive barrier schedule and preconditioned CG is the most efficient approach in this case.

Similar to the previous example, we use the randomly chosen \(10^{5}\) entries to compare \(P\) and \(P_{\G}\). We use \(P_{\G}\) obtained using the adaptive barrier schedule and preconditioned CG. On those points, the average relative entry-wise evaluation error is \(8.5 \times 10^{-3}\), which is approximately the entry-wise error from the first stage output \(\tilde{P}\). Therefore, we observe that the second-stage fitting does not introduce a significant fitting error compared to the first stage.

\paragraph{Example 3: Heavy-tail symmetric distribution}
We let \(n = 50\) and we let \((x_i)_{i=1}^{n}\) be a uniform grid point on \([0, 2]\) of length \(n\). We consider the unnormalized distribution tensor \(P \colon [n]^{d} \to \R\) defined as follows:
\begin{equation}\label{eqn: Cauchy kernel}
  P(i_1, \ldots, i_d) = p(x_{i_1}, \ldots, x_{i_d}), \quad p(z_1, \ldots, z_d) = \frac{1}{1 + z_1^2 + \cdots + z_d^2},
\end{equation}
where we take \(d = 30\).
The distribution \(p\) is the Cauchy distribution when \(d = 1\). The tensor in \Cref{eqn: Cauchy kernel} is different than the Ginzburg-Landau model in that the formula is symmetric in all of the variables.

In the first stage, we use the TT-cross algorithm using the analytic formula of \(P\) as defined in \Cref{eqn: Cauchy kernel}.
We use a maximal internal rank of \(r_{\max} = 20\) for \(\tilde{P}\).
For \(\tilde{P}\), the average relative entry-wise error on the randomly chosen \(10^{5}\) entries is \(6 \times 10^{-10}\), which shows that TT-cross is successful in compressing \(P\) with \(\tilde{P}\).

In the second stage, we similarly take \(\a = 20\) for the internal rank of \(P_{\G}\).
The numerical results for the heavy-tail symmetric tensor are shown in \Cref{fig:numerics-cauchy}.
We observe that \Cref{alg: alternating minimization} allows a fast decrease of the squared relative Frobenius error. In this case, the adaptive barrier schedule with full Newton has the best overall performance in wall-clock time. Moreover, the results show that using CG or preconditioned CG can provide acceleration in the initial stage.

Lastly, we report the relative entry-wise error on the randomly chosen \(10^{5}\) entries between \(P\) and \(P_{\G}\). We pick the \(P_{\G}\) obtained using the adaptive barrier schedule and preconditioned CG, and we report a loss of \(7.6 \times 10^{-6}\). The result shows that our procedure is successful in compressing \(P\) with an NTT ansatz.

\paragraph{Remark on CG performance}
We give a remark regarding the performance comparison between full Newton (using a direct solver to obtain the Newton direction) and CG. From \Cref{fig:numerics}, one can see that full Newton in some cases is more efficient than CG or preconditioned CG in reaching a high-accuracy solution. We remark that this is because a direct linear system solver is quite efficient for small linear systems. Therefore, when one chooses a relatively small internal rank \(\a\) (e.g. \(\a = 15\)), it might be beneficial to run full Newton.

\subsection{Density estimation}\label{sec: DE numerics}

\begin{figure}[!ht]
  \centering
    \subcaptionbox{Example 4, Ising model.\label{fig:numerics-Ising}}[0.76\linewidth]{
    \includegraphics[width=\linewidth]{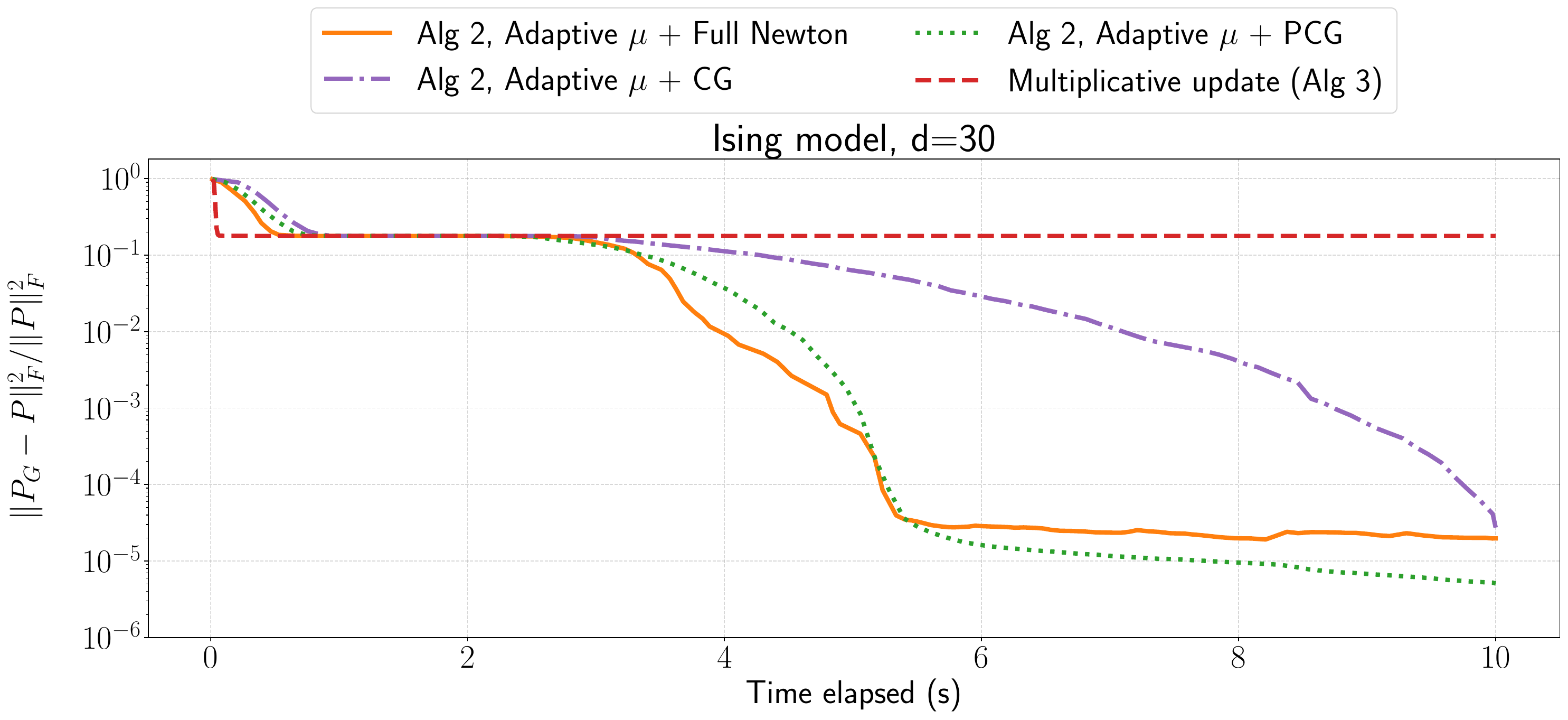}
  }
    \hfill
  \centering
  \subcaptionbox{Example 5, Heisenberg model.\label{fig:numerics-Heisenberg}}[0.75\linewidth]{
    \includegraphics[width=\linewidth]{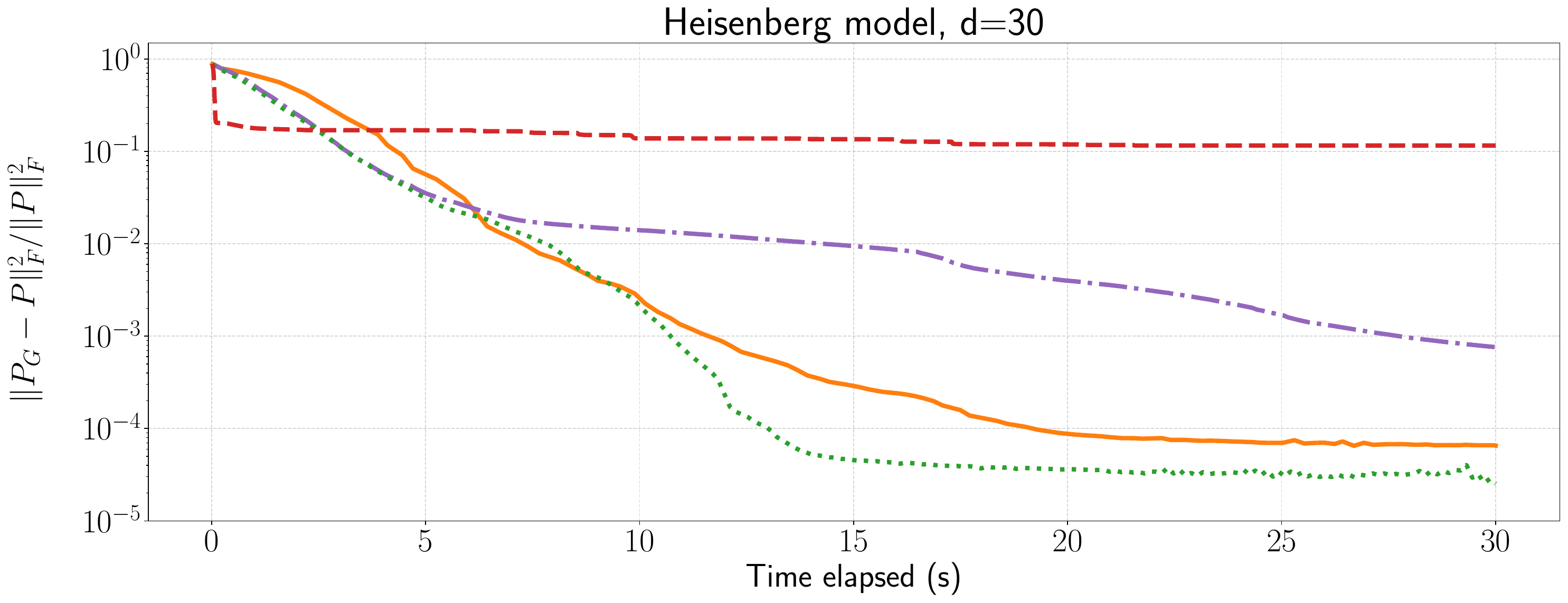}
  }

  \caption{Numerical results for NTT fitting in the density estimation setting. The details can be found in \Cref{fig:numerics}. For Example 4 we choose \(\sigma = 0.2\), and for Example 5 we choose \(\sigma = 0.005\). 
  }
  \label{fig:numerics_DE}
\end{figure}

The density estimation setting is subject to a larger noise level compared to the variational inference setting. One has samples \(\left(y_{1}^{(j)}, \ldots, y_{d}^{(j)}\right)_{j=1}^{N} \subset [n]^d\) which are distributed according to \(P\). The first stage uses the TT-sketch algorithm to compress \(P\) as a tensor train ansatz \(\tilde{P}\). In this case, the tensor \(\tilde{P}\) only converges to \(P\) at a Monte-Carlo rate \(\frac{1}{\sqrt{N}}\), and therefore the obtained tensor \(\tilde{P}\) contains many non-negligible negative entries. As the NTT fitting error cannot go to zero, our goal for NTT fitting is to reach below \(10^{-5}\) in the relative squared Frobenius loss. In this case, the fixed regularization strategy converges significantly slower than the adaptive barrier strategy; therefore, we omit it from the plot for simplicity.

To evaluate the accuracy of the obtained NTT model in fitting the given samples, we compare the negative-log-likelihood (NLL) of the obtained model \(P_{\G}\) against the ground truth model \(P\). The NLL of the NTT model \(P_{\G}\) is defined by 
\[
\mathrm{NLL}(P_{\G}) = \frac{1}{N}\sum_{j=1}^{N}\log\left(P_{\G}(y_{1}^{(j)}, \ldots, y_{d}^{(j)})\right) - \log(Z_{\G}),
\]
where \(Z_{\G} = \sum_{i_{[d]}}P_{\G}(i_{[d]})\) is the normalization constant of \(P_{\G}\). The NLL of \(P\) is defined likewise. The training is successful if the NLL of \(P_{\G}\) is close to that of \(P\). The TT-sketch output \(\tilde{P}\) cannot be used to report NLL because \(\tilde{P}\) often contains negative entries.

\paragraph{Example 4: 1D periodic Ising model}
We consider a ferromagnetic Ising model with periodic boundary conditions.
In this case, one has \(n = 2\) with \((x_1, x_2) = (-1, 1)\).
The distribution tensor \(P \colon [2]^{d} \to \R\) is defined by \[
  P(i_1, \ldots, i_d) \propto \exp{\left(\beta\sum_{ k - k' \equiv 1 \, \mathrm{mod} \,d}x_{i_k} x_{i_{k'}}\right)},\]
where the dimension is \(d = 30\) and the inverse temperature is \(\beta = 0.5\). We note that \(P\) is representable by a tensor train ansatz with rank \(r_{\mathrm{max}} = 4\). We simulate \(N = 5 \times 10^{5}\) samples with Monte-Carlo Markov chain methods using the analytic formula of \(P\).

In the first stage, we compress \(P\) by using the TT-sketch algorithm on the samples. For the second stage, we perform NTT fitting using a maximal rank of $\a = 10$ for the ansatz $P_{\G}$.
The results for NTT fitting are in \Cref{fig:numerics-Ising}.
One can see that our proposal is significantly more efficient than the multiplicative update algorithm, which is essentially stationary.
Moreover, the choice of adaptive barrier schedule and preconditioned CG is the most efficient approach. 

We test the performance of NTT compression using the NLL metric. The NLL of \(P_{G}\) is 17.467, and the NLL of \(P\) is 17.468. The result shows that the obtained model \(P_{\G}\) fits the generated samples as well as the ground truth model \(P\).

\paragraph{Example 5: 1D Heisenberg model}
Let \(d = 30\) be the dimension. We consider a probabilistic model arising from the ground state of a 1D antiferromagnetic Heisenberg model with \(d\) sites. Let \((\sigma^{X}_{k}, \sigma^{Y}_{k}, \sigma^{Z}_{k})\) be the Pauli matrices on site \(k\). We consider the Hamiltonian
\begin{equation}\label{eqn: heisenberg}
    H =\sum_{ k - k' = 1} \left(\sigma_{k}^{X} \sigma_{k'}^{X}+\sigma_{k}^{Y} \sigma_{k}^{Y}+\sigma_{k}^{Z} \sigma_{k'}^{Z}\right),
\end{equation}
and we obtain the ground state \(\psi \colon [2]^{d} \to \C\) by running DMRG with a maximal internal rank \(r_{\mathrm{max}} = 6\). From the obtained ground state \(\psi\), the distribution tensor \(P \colon [2]^{d} \to \R\) is defined by \[
  P(i_1, \ldots, i_d) = \lvert \psi(i_1, \ldots, i_d) \rvert^2,\]
where in particular \(P\) is normalized as \(\psi\) is a normalized vector in \(\C^{2^d}\). The distribution tensor \(P\) encodes the measurement outcome of \(\psi\) under computational basis measurement. We generate \(N = 8 \times 10^6\) samples of \(P\) by simulating computational basis measurement on \(\psi\).

In the first stage, we compress \(P\) by a tensor train \(\tilde{P}\) with rank \(r_{\mathrm{max}} = 10\). In the second stage, we use an NTT ansatz \(P_{\G}\) of internal rank \(\a = 20\), and the result is plotted in \ref{fig:numerics-Heisenberg}. One can see that the adaptive barrier, combined with preconditioned CG, yields the best result. In this case, the multiplicative update algorithm decreases in error, but the convergence is quite slow.

We use the output of preconditioned CG to compare the log likelihood. The NLL of \(P_{\G}\) is \(12.55\), and the NLL of \(P\) is \(12.53\). The result shows that the obtained model \(P_{\G}\) fits the generated samples roughly as well as the ground truth model \(P\). We note that the accuracy of the 1D Heisenberg model case is lower than that of the Ising model case. The reason for this lower accuracy is that the Heisenberg model is significantly more complex than the Ising model and requires a substantially larger number of samples to be modeled accurately.

\section{Discussion}\label{sec: conclusion}
We introduced a novel two-stage approach for variational inference and density estimation with the non-negative tensor train ansatz.
The approach enjoys fast convergence and is significantly more efficient than the multiplicative update algorithm derived from the Lee-Seung algorithm.
An interesting future direction is to utilize the NTT ansatz for the coefficient tensor in the functional tensor train ansatz \cite{chen2023combining}, which would enable the representation of a non-negative density function in the continuous setting. Lastly, we anticipate that the combination of the proposed approach with more sophisticated techniques can lead to further improvement in terms of training efficiency in the NTT fitting procedure.

\bibliographystyle{siamplain}
\bibliography{reference}

\end{document}